\documentclass[hidelinks,onefignum,onetabnum,final]{siamart220329}

\usepackage{lipsum}
\usepackage{amsfonts}
\usepackage{amssymb}
\usepackage{mathrsfs}
\usepackage{graphicx}
\usepackage{epstopdf}
\usepackage[compatible]{algpseudocode} 
\usepackage{xspace}
\usepackage{./mydef}

\ifpdf
  \DeclareGraphicsExtensions{.eps,.pdf,.png,.jpg}
\else
  \DeclareGraphicsExtensions{.eps}
\fi

\newcommand{\WL}[1]{{\color{black}{#1}}}

\newsiamremark{remark}{Remark}
\newsiamremark{hypothesis}{Hypothesis}
\crefname{hypothesis}{Hypothesis}{Hypotheses}
\newsiamthm{claim}{Claim}

\author{Wenyu Lei\footnotemark[2] \thanks{School of Mathematical Sciences, University of Electronic Science and Technology of China, No.2006, Xiyuan Ave, West Hi-Tech Zone, 611731, Chengdu, China.}
  \and George Turkiyyah\thanks{King Abdullah University of Science and Technology (KAUST), Thuwal, Saudi Arabia (\email{wenyu.lei, george.turkiyyah, omar.knio@kaust.edu.sa}).}
  \and Omar Knio\footnotemark[2]}

\usepackage{amsopn}

\usepackage{caption}
\usepackage{subcaption}

\graphicspath{{figures/}{tikz/}}

\headers{FEM for variable order FPDE}{Lei, Turkiyyah, and Knio}

\title{Finite element discretizations for \\ variable-order fractional diffusion problems}

\begin{document}

\maketitle

\begin{abstract}
  We present a finite element discretization scheme for multidimensional fractional diffusion problems with spatially varying diffusivity and fractional order. We consider the symmetric integral form of these nonlocal equations defined on general geometries and in arbitrary bounded domains.   A number of challenges are encountered when discretizing these equations. The first comes from the heterogeneous kernel singularity in the fractional integral operator. The second comes from the formally dense discrete operator with its quadratic growth in memory footprint and arithmetic operations. An additional challenge comes from the need to handle volume conditions--the generalization of classical local boundary conditions to the nonlocal setting. Satisfying these conditions requires that the effect of the whole domain, including both the interior and exterior regions, can be computed on every interior point in the discretization. Performed directly, this would result in quadratic complexity.  In order to address these challenges, we propose a strategy that decomposes the stiffness matrix into three components.  The first is a sparse matrix that handles the singular near-field separately, and is computed by adapting singular quadrature techniques available for the homogeneous case to the case of spatially variable order.
  The second component handles the remaining smooth part of the near-field as well as the far-field, and is approximated by a hierarchical  $\mathcal{H}^2$ matrix that maintains linear complexity in storage and operations.
  The third component handles the effect of the global mesh at every node, and is written as a weighted mass matrix whose
  density is computed by a fast-multipole type method.
  The resulting algorithm has therefore overall linear space and time complexity. Analysis of the consistency of the stiffness matrix is provided and numerical experiments are conducted to illustrate the convergence and performance of the proposed algorithm.
\end{abstract}

\begin{keywords}
  fractional diffusion, variable order, finite element approximation, hierarchical low-rank approximations, fast multiple method.
\end{keywords}

\begin{MSCcodes}
  65N22, 
  65N30, 
  65N38, 
  65N06  
\end{MSCcodes}




\section{Introduction}
\label{sec:intro}

Fractional diffusion equations are becoming increasingly important in modeling phenomena where nonlocal effects are significant, including fractured media in material science \cite{silling2000reformulation,suzuki2022fractional}, transport in complex media \cite{SONG2016,WANG2019}, stable L\'evy processes in finance \cite{tankov2003financial,levendorskii2004pricing}, Gaussian random fields in spatial statistics \cite{lindgren2011explicit,lindgren2022spde}, image denoising \cite{gatto2015numerical}, among many others. In addition to physics-based modeling, fractional operators  have also been used in controlling the smoothness of priors in Bayesian inverse problems involving distributed parameters \cite{chen17}. While there has been much work devoted to the formulation and discretization of fractional diffusion in the homogenous case with constant fractional order and material properties (diffusivity, permeability, etc.), there has been comparatively little work that addresses the practically useful heterogenous case with variable coefficients, particularly in multiple spatial dimensions. Our motivation in this work is to develop methods for the efficient discretization of the fractional diffusion operator with variable fractional order and properties.

A standard formulation of fractional diffusion involves the integral representation of the fractional Laplacian $(-\Delta)^s$ of order $s\in(0,1)$, where the fractional order characterizes the global smoothness of the solution \cite{grubb2015fractional,borthagaray2022besov,faustmann2022weighted}.
The model may be extended to the variable order and diffusivity case by defining these quantities as functions in a physical domain; see e.g. \cite{contreras2016multi,d2022fractional,glusa2022asymptotically,suzuki2022fractional}. The specific form of the variable-order fractional diffusion operator we consider here is defined in a bounded domain $\Omega\subset \mathbb R^d$ with $d=1,2,3$ as:
\begin{equation}\label{e:op}
    \mathcal L u(x):= \text{p.v.}\int_\Omega  \frac{a(x,y) \left( u(x) - u(y) \right)}{|x-y|^{d+s(x)+s(y)}}\diff y ,
\end{equation}
where $a(x,y) : \Omega\times \Omega\to \mathbb R^+$ denotes the diffusion coefficient and $s(x): \mathbb R^d\to (0,1)$ is the variable-order function. Formal settings for the above operator with appropriate volume constraints to insure well-posedness are described in Section~\ref{sec:formulation} below.

The operator \eqref{e:op} becomes the classical integral fractional Laplacian operator (or the Riesz potential) when $s(x)$ and $a(x,y)$ are constant functions and $\Omega=\mathbb R^d$.
The numerical approximation and analysis of this integral fractional Laplacian have been extensively studied in the literature. We refer to \cite{huang2014numerical,duo2018novel,hao2021fractional,minden2020simple,duo2019accurate} for finite difference methods, \cite{XU2020} for a collocation approach,
and  \cite{boukaram2020hierarchical,luchko2016new,zhang2007numerical,gorenflo2007continuous} for particle methods as well as random walk approaches. In terms of finite element methods, we refer to \cite{acosta2017fractional,ainsworth2017aspects} for a classical conforming scheme, \cite{d2013fractional} based on volume constraints, \cite{bonito2019numerical} for a non-conforming approach by the Dunford-Taylor integral formulation, \cite{WANG2015,LIAN2016} for Petrov-Galerkin approaches, and \cite{karkulik2019h} based on the well-known Caffarelli-Silvestre extension (cf. \cite{caffarelli2007extension}).

In contrast, the numerical study for the variable-order fractional operator \eqref{e:op} has just started in recent years. For one dimensional problems, we refer to \cite{ZHAO2019,zheng2020optimal,jia2021fast,alzahrani21} for the numerical approximations of different formulations of variable-order fractional diffusion operators. In multiple spatial dimensions, finite difference schemes were presented in \cite{alzahrani2022space} for discretizing the operator \eqref{e:op} in Cartesian geometries. A finite element approach is presented in \cite{d2022fractional} for domains where the variability consists of uniform inclusions in an otherwise homogeneous medium. There is still however no complete treatment of general variable order and properties for arbitrary geometries in the multidimensional setting.
In this work, we extend the considerations from the finite-difference method on cartesian grids developed in \cite{alzahrani2022space} to the quasi-uniform meshes for general polytopes, and approximate variable-order fractional diffusion problems involving \eqref{e:op} using a linear finite element discretization scheme.

A primary challenge in the finite element discretization of \eqref{e:op} is that the non-local fractional kernel results in a dense stiffness matrix with a prohibitive $O(N^2)$ quadratic growth in memory footprint and arithmetic operations, where $N$ denotes the number of degrees of freedom. In \cite{ainsworth2017aspects}, the stiffness matrix is decomposed into the near-field and the far-field based on  interactions between element pairs and an $\mathcal H^2$ approximation is constructed from a kernel interpolation \cite{hackbusch2002h2}. The corresponding near-field matrix is sparse while the far-field matrix can be compressed as a hierarchically low-rank matrix, thus reducing the complexity of the assembly process to $O(N)$, with constants depending on the order of the quadrature and polynomial approximation of the kernel. We also refer to \cite{ZHAO2017,karkulik2019h,boukaram2020hierarchical,Li2020,bauer2022kernel} for related hierarchical low-rank approximations, also all in the context of constant-order fractional problems.
The variable coefficient case, however, requires additional considerations that are not present in the constant coefficient setting. For example, the work in \cite{ainsworth2017aspects,ainsworth2018towards} takes the advantage of the constant-order property to evaluate the
integral fractional Laplacian locally; see \cite[Lemma~4]{ainsworth2017aspects}, a procedure that cannot be readily extended to the variable-order case.

In our approach, the non-singular far-field matrix is expressed as the sum of two terms, where the first is a matrix that fits the structure for the $\mathcal H^2$ approximation,
\begin{equation}\label{e:far}
    \int_\Omega\int_\Omega V(x) \gamma_{\mathcal T}(x,y) U(y) \diff y\diff x,
\end{equation}
where $U$ and $V$ are finite element functions and $\gamma_{\mathcal T}$ is a de-singularized kernel depending on the subdivision $\mathcal T$ of $\Omega$. The second term is a sparse \emph{weighted mass matrix} with the density function $\rho_{\mathcal T}(x) = \int_\Omega\gamma_{\mathcal T}(x,y)\diff y$. We compute this weighted mass matrix by quadrature on each element. The bottleneck of this procedure is that a direct evaluation of the density function at each quadrature point is $O(N)$, resulting in a quadratic overall complexity.  To remedy this growth, we propose a fast multipole method to obtain the values at all the quadrature points in linear time complexity.

Another challenge of \eqref{e:op} is that the singularity of the kernel requires special quadrature rules in order not to slow down the rate of convergence as the discretization is refined. To address this challenge, we adapt techniques from boundary element computations (see e.g. \cite{sauter2010boundary}) to resolve the singularities of the integrand when computing the entries in element stiffness matrices.  Thus, we extend the techniques developed for the constant-order problem \cite{acosta2017short,acosta2017fractional} and achieve a treatment similar to that in \cite{alzahrani2022space} which
relied on a singularity subtraction technique that is particularly convenient to express on a cartesian finite difference grid.

The main contributions of this work are:
\begin{itemize}
    \item We propose a decomposition of the stiffness matrix for the finite element approximation of \eqref{e:op} into three sub-matrices. All three sub-matrices can be computed in $O(N)$, with construction criteria depending on the distance between shape function pairs and element pairs in the subdivision. This strategy allows the handling of general domains with spatially-variable order and coefficients as well bounded exterior regions.

    \item We generalize the variable-order setting discussed in \cite{d2022fractional} from element-wise constant functions to element-wise analytic functions. The numerical integration techniques for singular integral kernels developed by \cite{acosta2017short,acosta2017fractional} are generalized as well. The quadrature error is shown to converge at rates controllable simply by the order of the quadrature used. We also note that the proposed computation techniques can be applied to more general nonlocal kernels introduced by \cite{du2012analysis}.

    \item We propose an $\mathcal H^2$ representation to approximate the form \eqref{e:far} and a variable-kernel fast multipole method for  computing the density $\rho_{\mathcal T}(x)$ at quadrature points. The variable-kernel fast multipole computation is cast as a matrix vector product, where the matrix is constructed as an $\mathcal H^2$ matrix via a collocation method and the vector consists of quadrature weights. The errors in these approximations are shown to converge at rates readily controllable by the order of the polynomial used in the kernel approximation.
\end{itemize}

The rest of this paper is organized as follows.
Section~\ref{sec:formulation} presents a weak formulation of the problem and its finite-element discretization.
Section~\ref{sec:decomposition} describes the decomposition of the stiffness matrix into three sub-matrices that handle different aspects of the problem.
\Cref{s:direct,s:hmatrix,s:fmm} discuss the numerical approximation of these individual sub-matrices and analyze the resulting discretization errors.
Specifically, Section~\ref{s:direct} introduces a tensor-product quadrature rule to directly compute the singular near-field interactions.
Section~\ref{s:hmatrix} presents an ${\cal H}^2$ representation of the sub-matrix that approximates the non-singular part of the near-field as well as the far-field.
Section~\ref{s:fmm} describes the fast multipole machinery to efficiently account for global interactions at quadrature points.
Section~\ref{s:numerics} presents numerical experiments to illustrate the convergence of the finite element approximation as well as the linear-complexity of the proposed assembly process.
Section~\ref{sec:conclusions} concludes with directions for future work.

\emph{Notation.} In the following, we set $a\lesssim b$ if $a\le Cb$ with $C$ denoting a positive constant independent of $a$, $b$, and the discretization parameters (e.g. the mesh size $h$, the number of degrees of freedom $N$, the quadrature order $n$, and the polynomial degree $p$). We set $a\sim b$ when $a\lesssim b$ and $b\lesssim a$.


\section{Weak formulation and finite element discretization}
\label{sec:formulation}
Our weak formulation for the nonlocal symmetric operator \eqref{e:op} starts from the definition of a generalized nonlocal divergence operator, $\mathcal D$, introduced in \cite{du2013nonlocal}. For a vector field $v: \mathbb R^d\times \mathbb R^d\to \mathbb R$,
\[
    \mathcal D[v](x) = \int_{\mathbb R^d} (v(x ,y)+ v(y, x)) \cdot \alpha(x , y)\diff y .
\]
Here the vector field $\alpha(x,y) : \mathbb R^d \times \mathbb R^d \to \mathbb R^d$ satisfies the antisymmetric property $\alpha(x,y) = -\alpha(y,x)$. An adjoint operator $\mathcal D^*$ corresponding to $\mathcal D$ under the $L^2(\mathbb R^d)$ inner product may be written as
\[
    \mathcal D^*[u] (x,y) := - (u(y) - u(x)) \alpha(x, y),
\]
with $-\mathcal D^*$ interpreted as a nonlocal gradient.

Similarly to the classical diffusion operator which is defined with a second-order diffusion coefficient $a$ (a symmetric tensor representing diffusivity, permeability, or related material properties), the nonlocal diffusion operator can be defined as
\begin{equation}\label{e:nl-diff}
    \mathcal L u:= \mathcal D[a\cdot \mathcal D^*[u]](x) := -2 \int_{\mathbb R^d} (u(y)-u(x))\gamma (x,y) \diff y
\end{equation}
with
\[
    \gamma := \alpha(x,y)^T a(x,y) \alpha(x,y) .
\]
In order to simplify our discussion, we set $a$ to be diagonal, and assume that there exists an analytic function $\kappa(x)$ so that $\kappa(x)\ge \delta >0$ for some positive constant $\delta$, and that $a(x,y) = \sqrt{\kappa(x)\kappa(y)} I$ with $I\in \mathbb R^{d\times d}$ denoting the identity matrix
\footnote{Even though we consider globally defined kernels here, our treatment also applies to finite horizon problems where $a$ vanishes when $|x-y|$ exceeds an interaction distance threshold.}.
We are interested in the numerical approximation of the diffusion operator $\mathcal L$ with $\alpha$ defined by
\[
    \alpha(x,y) = \frac{y-x}{|y-x|^{\tfrac{d+s(x)+s(y)}2 + 1}} ,
\]
where $s(x)$ is an analytic function satisfying that
\begin{equation}\label{i:s}
    0\le \underline s\le s(x)\le \overline s < 1 .
\end{equation}
Plugging the above definition of $\alpha$ into \eqref{e:nl-diff} yields
\[
    \mathcal Lu(x) = -2 \int_{\mathbb R^d} \frac{a(x,y)(u(x)-u(y))}{|x-y|^{d+s(x)+s(y)}} \diff y,
\]
a spatially variable-order generalization of the integral fractional Laplacian.

\subsection{Volume-constrained problems}
We shall apply the operator $\mathcal L$ in a bounded domain $\Omega:=\overline\Oi\cup \Oe\subset \mathbb R^d$ with $\Oi\cap\Oe = \emptyset$. Here the domain $\Oi$ has Lipschitz boundary and $\Oe$ is an enclosing region of $\Oi$ satisfying that $0< |\Oe|<\infty$.
Unlike classical, second-order elliptic problems, where imposing boundary conditions on $\partial \Oi$ is sufficient to guarantee well-posedness, a general non-local fractional operator requires that volume constraints be imposed on $\Oe$ \cite{du2013nonlocal}
\footnote{We limit our discussion to $|\Oe| > 0$ as this guarantees well-posedness of the Dirichlet problem. If $\Omega = \Oi$, then the problem may not be well-posed when $\underline{s} \le 1/2$ (cf. \cite{fall2022regional}).}.
These conditions are in a sense the equivalent of the boundary conditions of the local operator. Here we focus on the Dirichlet volume-constrained problem which can be stated as: given a data function $f$ supported in $\Oi$, we seek $u \in \Omega$ satisfying
\begin{equation}\label{e:strong}
    \begin{aligned}
        - \mathcal L_\Omega u & = f, &  & \text{ in } \Oi,  \\
        u                     & =0 , &  & \text{ in } \Oe ,
    \end{aligned}
\end{equation}
with the operator
\[
    \mathcal L_\Omega u =  -2 \int_{\Omega} \frac{a(x,y)(u(x)-u(y))}{|x-y|^{d+s(x)+s(y)}} \diff y .
\]
\subsection{Weak formulation}
Define the energy space
\[
    \mathbb V := \{v\in L^2(\Omega) : \|v\|_{\mathbb V}<\infty \text{ and } v = 0 \text{ in } \Oe\},
\]
where the energy norm $\|.\|_{\mathbb V}$ is given by $\|v\|_{\mathbb V}^2 := \|v\|_{L^2(\Omega)}^2 + |v|_{\mathbb V}^2$ with
\[
    |v|_{\mathbb V}^2 := \int_{\Omega}\int_{\Omega}
    (u(x) - u(y))^2\gamma(x,y) \diff y\diff x \quad\text{and}\quad
    \gamma(x,y) = \frac{a(x,y)}{|x-y|^{d+s(x)+s(y)}} .
\]
Clearly, the energy space $\mathbb V$ is a Hilbert space. By utilizing a standard argument (e.g., Proposition~2.2 and 2.4 in \cite{acosta2017fractional}), we can show the following Poincar\'e inequality
\begin{equation}\label{i:poincare}
    \|v\|_{L^2(\Omega)} \lesssim |v|_{\mathbb V}, \quad\text{for } v\in \mathbb V .
\end{equation}
Multiplying the first equation of \eqref{e:strong} with a test function $v\in \mathbb V$, and integrating the resulting equation over $\Omega$, we have
\[
    \begin{aligned}
        \int_\Omega L_{\Omega} u v\diff x & = -2\int_{\Omega}\int_{\Omega}(u(x)-u(y))\gamma(x,y) v(x) \diff y \diff x     \\
                                          & = -2\int_{\Omega}\int_{\Omega} (u(y)-u(x))\gamma(x,y) v(y)  \diff y \diff x . \\
    \end{aligned}
\]
To obtain the second equality we switched the order of the double integral and used the fact that $\gamma(x,y) = \gamma(y,x)$. Summing up the two double integrals on the right hand side leads to the definition of the bilinear form
\[
    \mathcal A(u,v) := -\int_{\Omega} L_\Omega u v\diff x
    =\int_{\Omega}\int_\Omega (u(x)-u(y))(v(x)-v(y))\gamma(x,y) \diff y\diff x .
\]
So our weak formulation reads: find $u\in\mathbb V$ satisfying that
\begin{equation}\label{e:weak}
    \mathcal A(u,v) = \int_{\Oi} f v\diff x \quad \forall v\in \mathbb V .
\end{equation}
Clearly, $\mathcal A(u,v)\le \|u\|_{\mathbb V}\|v\|_{\mathbb V}$ by the Cauchy-Schwarz inequality. So $\mathcal A(.,.)$ is bounded in $\mathbb V$. Thanks to \eqref{i:poincare}, we obtain the coercivity of the bilinear form. Thus, the Lax-Milgram lemma guarantees that the weak formulation \eqref{e:weak-d} admits a unique solution $u\in\mathbb V$.

\subsection{Finite element discretization}
We consider a simplicial finite element mesh that subdivides the interior and exterior polytope regions $\Oi$ and $\Oe$. We denote by $\mathcal T:=\mathcal T(\Omega)$ the resulting subdivision of $\Omega$, and assume that $\mathcal T(\Omega)$ matches the boundary of $\Oi$, \ie $\overline\Oi\cap\overline\Oe$ consists of faces from $\mathcal T$, allowing us to identify
$\mathcal{T}^\text{int}$ and $\mathcal{T}^\text{ext}$ as the subdivisions of the interior and exterior regions.
Let $\mathcal N := \{\bx_i\}_{i=1}^M$ be the set of vertices associated with the subdivision $\mathcal T$ . In particular, we set the first $N$ ($N<M$) nodes to be the interior nodes in $\Oi$ and denote the collection of them to be $\mathcal N^\circ$. Also, we let $\mathcal N^c:=\mathcal N\backslash \mathcal N^\circ$.
Let $\mathbb V(\mathcal T)\subset \mathbb V$ be the conforming continuous piecewise linear finite element space associated with $\mathcal T$.
For each node $\bx_i\in \mathcal N$, we set $\psi_i \in \mathbb V(\mathcal T)$ to be the corresponding linear shape function. The above notations allow us to write the discrete solution as $U= \sum_{j=1}^N u_j \psi_j$ with $\underline U = (u_1,u_2,\ldots, u_N)^T\in \mathbb R^N$ and express the test function as $V = \sum_{i=1}^N v_i\psi_i$ with the coefficient vector $\underline V = (v_1,v_2,\ldots, v_N)^T\in \mathbb R^N$.
We also denote the support of $\psi_i$ by $\mathcal S_i$, and the patch of each cell $\tau \in \mathcal T$ by $\mathcal S_\tau$, namely (see Figure~\ref{fig:patch})
\[
    \mathcal S_i := \bigcup_{\tau \in {\mathcal T} | \bx_i\in \tau} \tau, \qquad
    \mathcal S_\tau := \bigcup_{ \substack{\tau' \in \mathcal T \\ \tau\cap \tau'\neq \emptyset} } \tau' \, .
\]
\begin{figure}[ht]
    \begin{center}
        \includegraphics[width=.22\textwidth]{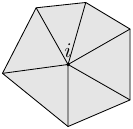}
        \hspace*{1.5cm}
        \includegraphics[width=.25\textwidth]{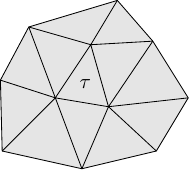}
    \end{center}
    \caption{Support $\mathcal S_i$ of basis function $\psi_i$ and patch $\mathcal S_\tau$ of cell $\tau$.}
    \label{fig:patch}
\end{figure}

A finite element discretization with respect to \eqref{e:weak} seeks $U\in \mathbb V(\mathcal T)$ so that $U = 0$ in $\Oe$ and
\begin{equation}\label{e:weak-d}
    \mathcal A(U,V) = \int_{\Oi} f V\diff x \quad \forall V\in \mathbb V(\mathcal T) .
\end{equation}

Because the above numerical scheme is conforming, namely $\mathbb V(\mathcal T)\subset \mathbb V$, the well-posedness follows from the continuous formulation \eqref{e:weak}.
Using the finite element mesh, we can rewrite the discrete bilinear form in \eqref{e:weak-d} as
\[
    \mathcal A(U,V) = \sum_{\tau,\tau'\in\mathcal T} \mathcal A_{\tau,\tau'}(U,V),
\]
where
\begin{equation}\label{e:l-stiff}
    \mathcal A_{\tau,\tau'}(U,V) := \int_\tau\int_{\tau'} (U(x)-U(y))(V(x)-V(y))\gamma(x,y)\diff y\diff x .
\end{equation}
The stiffness matrix assembly, \ie the construction of the matrix $\underline A \in \mathbb R^{N\times N}$ with entries $\underline A_{ij} = \mathcal A(\psi_j,\psi_i)$, is based on contributions of element pairs ($\tau, \tau'$) from \eqref{e:l-stiff}, with each such contribution involving interactions \eqref{e:l-stiff} between shape functions that are supported on either $\tau$ or $\tau'$. In other words, if we let $\mathcal I(\tau,\tau')\subset\{1,\ldots,N\}$ be the union of the global index sets of the vertices of elements $\tau$ and $\tau'$ so that for $i \in \mathcal I(\tau,\tau')$, we have $\tau \subset \mathcal S_i$ or $\tau' \subset \mathcal S_i$. Then we can define the local stiffness matrix of an element pair as
\begin{equation} \label{eq:elem}
    \underline A^{\tau,\tau'} = \{\mathcal A_{\tau,\tau'}(\psi_j,\psi_i)\}_{i,j\in \mathcal I(\tau,\tau')}
\end{equation}
and assemble these elemental contributions into a global stiffness matrix $\underline A$.


\section{Decomposition of the stiffness matrix}
\label{sec:decomposition}

A direct assembly of the elemental stiffness matrices $\underline A^{\tau,\tau'}$  of \eqref{eq:elem} into a global stiffness matrix $\underline A$ will obviously result in a scheme with quadratic complexity in both storage and operations. In this section, we outline a splitting scheme that decomposes the stiffness matrix into three components, in order to obtain a linear complexity algorithm. The splitting reflects distinct computational characteristics of the problem, and the resulting decomposition is primarily motivated by the different linear-complexity construction algorithms for the three matrix components, as we describe in detail in sections \ref{s:direct}, \ref{s:hmatrix}, and \ref{s:fmm}, respectively.

Our consideration for the decomposition of $\underline A$ starts with the relation between two elements $\tau,\tau'\in \mathcal T$. Our goal is to write the bilinear form in \eqref{e:weak-d} as the sum:
\begin{equation}\label{e:decomp}
    \mathcal{A}(U,V) = \!
    \sum_{\substack{\tau \in \mathcal{T}^\text{int}, \tau'\in \mathcal{T} \\ \tau \cap \tau' \neq \emptyset}} \! \!
    \mathcal{A}_{\tau,\tau'}(U,V) +
    \sum_{\substack{ \tau,\tau'\in \mathcal{T}^\text{int} \\  \tau \cap \tau' = \emptyset}}  \! \! \mathcal{K}_{\tau,\tau'}(U,V) +
    \sum_{\substack{\tau \in \mathcal{T}^\text{int}, \tau'\in \mathcal{T} \\ \tau \cap \tau' = \emptyset}} \! \!
    \mathcal{M}_{\tau,\tau'}(U,V).
\end{equation}
The first term of the sum handles the singular integrals that arises when $\tau$ and $\tau'$ are not separated. When $\tau$ and $\tau'$ are separated, the bilinear form $\mathcal A_{\tau,\tau'}(., .)$, which is now non-singular, can be split into two components: $\mathcal K_{\tau,\tau'}(., .)$ which represents effects of $\tau$ on all non-neighboring elements of the mesh, and
$\mathcal M_{\tau,\tau'}(., .)$ which represents effects from all non-neighboring elements of the mesh (including elements in $\mathcal{T}^\text{ext}$)  on $\tau$. Formal definitions for these two forms will be given later in \eqref{eq:sum}. \WL{We will show that under suitable mesh setting the matrix associated with the frist tem truns out to be sparse, while} hierarchical matrix algorithms and fast multipole methods allow us to construct the last two terms of \eqref{e:decomp} in linear complexity.

\subsection{Case \texorpdfstring{$\tau \cap \tau' \neq \emptyset$}{}}
When $\tau$ and $\tau'$ are direct neighbors or $\tau = \tau'$, we define the corresponding near-field bilinear form as
\begin{equation}\label{e:b}
    \begin{aligned}
        \mathcal B(U,V) & :=
        \WL{\sum_{\substack{\tau \in \mathcal{T}^\text{int}, \tau'\in \mathcal{T} \\ \tau \cap \tau' \neq \emptyset}}}
        \mathcal A_{\tau,\tau'}(U,V)
                        & =\sum_{\tau\in\mathcal{T}^\text{int}} \sum_{\tau'\subset \mathcal S_{\tau}} \mathcal A_{\tau,\tau'}(U,V) .
    \end{aligned}
\end{equation}
Here we note that the elements $\tau$ in the above form are restricted to those in $\Oi$ since both $U$ and $V$ vanish in $\Oe$. Denoting the near-field matrix $\underline B \in \mathbb R^{N\times N}$ with entries $\underline B_{ij} = \mathcal B(\psi_j,\psi_i)$, we assemble it directly from element-pair local stiffness matrices $\underline A^{\tau, \tau'}$ \WL{and call this process $\texttt{DirectEval}(\tau,\tau')$}.
The primary challenge here is the accurate evaluation of these local stiffness matrices which involves singular integrands with spatially-varying fractional order. As we describe in detail in Section~\ref{s:direct}, we generalize ideas from boundary element methods for transferring the singular integrands in \eqref{e:l-stiff} to analytical integrands (cf. \cite{acosta2017short}) to handle the case of variable order.

We further assume that the triangulation $\mathcal T$ is shape-regular and quasi-uniform, \ie there exist two positive constants $c_{\texttt{sr}}$ and $c_{\texttt{u}}$ so that for all $\tau\in\mathcal T$ and $\tau\subset \Omega$, there holds that
\[
    h_\tau \le c_{\texttt{sr}} \rho_\tau\quad\text{ and }\quad \max_{\tau\in\mathcal T}h_\tau \le c_{\texttt{u}} \min_{\tau\in \mathcal T} h_\tau ,
\]
with $h_\tau$ and $\rho_\tau$ denoting the size of $\tau$ and the maximum size of the inscribed ball in $\tau$. Thus $\underline B$ is a sparse matrix and the maximum number of the nonzero column entries, namely
\[
    \max_{i=1,\ldots, N}\#\{\bx_j \in \overline {\mathcal S}_\tau, \tau\subset\mathcal S_i\}
\]
is uniformly bounded and only depends on $c_{\texttt{sr}}$, insuring linear complexity.

\subsection{Case \texorpdfstring{$\tau \cap \tau' = \emptyset$}{}}
When the elements $\tau$ and $\tau'$ are separated, $|x - y| > 0$ for all $x \in \tau$ and $y \in \tau'$, the singularity of the integrand is avoided but another difficulty is introduced because of the quadratic number of the element pairs that have to be considered. A different strategy is hence needed for constructing this contribution to the stiffness matrix.

We start by observing that if both $\tau,\tau'\in \Oe$, we immediately get $\mathcal A_{\tau,\tau'}(\psi_j,\psi_i)=0$. We can therefore fix $\tau\in \mathcal T^{\text{int}}$ and further consider this case by whether $\tau'$ is located in $\Oi$ or in $\Oe$.

\subsubsection{Subcase \texorpdfstring{$\tau, \tau' \in \mathcal{T}^\text{int}$}{}}

Since the $\mathcal A_{\tau,\tau'}(.,.)$ is not singular, we can rewrite it as
\begin{align}
    \mathcal A_{\tau,\tau'}(U,V) & = \int_{\tau}\int_{\tau'}
    \bigg(U(x)V(x)+U(y)V(y)-U(x)V(y)-U(y)V(x)\bigg)
    \gamma(x,y)\diff y\diff x   \nonumber                                                                                                                                                       \\
                                 & = \bigg[\int_\tau U(x)V(x) \int_{\tau'} \gamma(x,y)\diff y\diff x  + \int_{\tau'} U(x)V(x) \int_{\tau}\gamma(x,y) \diff y\diff x\bigg]   \nonumber           \\
                                 & \quad\quad- \bigg[\int_\tau U(x) \int_{\tau'} \gamma(x,y)V(y)\diff y\diff x  + \int_{\tau'} U(x) \int_{\tau}\gamma(x,y)V(y) \diff y\diff x \bigg]  \nonumber \\
                                 & =:\mathcal{M}_{\tau,\tau'}(U,V) + \mathcal{K}_{\tau,\tau'}(U,V) \label{eq:sum}
\end{align}

We sum $\mathcal{K}_{\tau,\tau'}$ for all elements $\tau,\tau'$ in $\Oi$. By the symmetry of choosing between $\tau$ and $\tau'$, we can derive that
\begin{equation}\label{e:k}
    \begin{aligned}
        \mathcal K(U,V) & := \sum_{\tau\in \mathcal{T}^\text{int}}\sum_{\tau'\in \mathcal{T}^\text{int}} \mathcal{K}_{\tau,\tau'}(U,V) \\
                        & = - 2 \int_{\Oi}\int_{\Oi}\gamma_{\mathcal T}(x,y) U(x)V(y) \diff y\diff x ,
    \end{aligned}
\end{equation}
where $\gamma_{\mathcal T}$ is a mesh-dependent kernel
\begin{equation}\label{e:g-t}
    \gamma_{\mathcal T}(x,y) := \left\{
    \begin{aligned}
         & 0,\quad           &  & x\in \tau, y\in \tau' \text{ for some }\tau,\tau'\in \mathcal T \text{ with } \tau \cap \tau'\neq \emptyset , \\
         & \gamma(x,y),\quad &  & \text{otherwise} .                                                                                            \\
    \end{aligned}
    \right.
\end{equation}
This mesh-dependent kernel evaluates to zero precisely in the integration regions that have already been handled by the first case above, i.e., when $\tau \cap \tau' \neq \emptyset$.

Following a similar argument, we sum $\mathcal{M}_{\tau,\tau'}$ for all $\tau,\tau'\subset\Oi$ to define
\begin{equation}\label{e:m-int}
    \begin{aligned}
        \cMi (U,V) & :=\sum_{\tau\in \mathcal{T}^\text{int}} \sum_{\tau'\in \mathcal{T}^\text{int}}\mathcal{M}_{\tau,\tau'}(U,V)                                               \\
                   & = 2\sum_{\tau\in \mathcal{T}^\text{int}} \int_\tau U(x)V(x) \bigg(\int_{\Oi\backslash\mathcal S_\tau} \gamma(x,y) \diff y\bigg) \diff x . \\
    \end{aligned}
\end{equation}

\subsubsection{Subcase \texorpdfstring{$\tau \in \mathcal{T}^\text{int}$, $\tau' \in \mathcal{T}^\text{ext}$}{}}
In this case, we can simplify the form $\mathcal A_{\tau,\tau'}(.,.)$ knowing that $U(y)=V(y)=0$ for $y\in \tau'$ due to $U,V\in \mathbb V(\mathcal T)$. The second term in \eqref{eq:sum} vanishes and $\mathcal A_{\tau,\tau'}(.,.)$ is simply $\mathcal{M}_{\tau,\tau'}(U,V)$, which we sum for all $\tau\in\mathcal{T}^\text{int}$ and $\tau'\in\mathcal{T}^\text{ext}$ to obtain
\begin{equation}\label{e:m-ext}
    \begin{aligned}
        \cMe(U,V) & := 2\sum_{\tau\in \mathcal{T}^\text{int}}\sum_{\tau'\in \mathcal{T}^\text{ext}} \int_\tau U(x)V(x) \bigg(\int_{\tau'} \gamma(x,y) \diff y\bigg) \diff x \\
                  & = 2\sum_{\tau\in \mathcal{T}^\text{int}} \int_\tau U(x)V(x) \bigg(\int_{\Oe\backslash\mathcal S_\tau} \gamma(x,y) \diff y\bigg) \diff x.
    \end{aligned}
\end{equation}
Here we also have the factor 2 in the above equation because we accounted for the symmetric case $\tau'\in \mathcal{T}^\text{int}$ and $\tau\in \mathcal{T}^\text{ext}$.

Gathering \eqref{e:m-int} and \eqref{e:m-ext}, we can define the form
\begin{equation}\label{e:m}
    \begin{aligned}
        \mathcal M(U,V) & := \cMe(U,V) + \cMi(U,V)                                                                                                    \\
                        & =  2\sum_{\tau\in \mathcal{T}^\text{int}} \int_\tau U(x)V(x) \bigg(\int_{\Omega\backslash\mathcal S_\tau} \gamma(x,y) \diff y\bigg) \diff x \\ &=  2\sum_{\tau\in \mathcal{T}^\text{int}} \int_\tau U(x)V(x) \rho_{\mathcal T}(x)\diff x ,
                        &
    \end{aligned}
\end{equation}
where
\begin{equation}
    \rho_{\mathcal T}(x) := \int_{\Omega\backslash\mathcal S_\tau} \gamma(x,y) \diff y = \int_{\Omega} \gamma_{\mathcal T}(x,y) \diff y
\end{equation}
recalling that $\gamma_{\mathcal T}$ is defined in \eqref{e:g-t}. In this form, it is easy to see that a matrix $\underline M$ with entries $\underline M_{i,j}=\mathcal M(\psi_j,\psi_i)$ for $i,j=1,\ldots, N$ has the footprint of a mass matrix and may be assembled from element contributions. If $i$ and $j$ are the indices for the shape functions defined on $\tau$, an element stiffness matrix is written as:
        \begin{equation}\label{e:mass-l}
            \underline M^\tau = \left\{\int_\tau \psi_j(x)\psi_i(x)
            \rho_{\mathcal T}(x) \diff x\right\}_{ij}.
        \end{equation}
$\underline M$ is in fact a weighted mass matrix with $\rho_{\mathcal T}(x)$ playing the role of a density function. The evaluation of $\rho_{\mathcal T}(x)$ is required at all quadrature points, with each evaluation requiring a global integration. This appears to demand quadratic complexity in arithmetic operations. However, a method akin to a fast multipole method for spatially-varying kernels can evaluate $\rho_{\mathcal T}(x)$ at all quadrature points in linear complexity as we show in Section \ref{s:fmm}.

\subsection{Assembly of complete stiffness matrix}
The bilinear forms in \eqref{e:b}, \eqref{e:k}, \eqref{e:m} provide the three components, $\underline B$, $\underline K$, and $\underline M$ of the stiffness matrix $\underline A$ of the problem, where $\underline K_{i,j} = \mathcal K(\psi_j,\psi_i)$ for $i,j=1,\ldots,N$. The matrices $\underline B$ and $\underline M$ are sparse and can be stored directly. The matrix $\underline K$ however is formally dense. We take advantage of the structure of the bilinear form $\mathcal K(.,.)$ with its de-singularized kernel to store $\underline K$ in linear complexity using $\mathcal H^2$-matrix compression techniques \cite{hackbusch2002h2}. We also use construction algorithms that approximate the smooth kernel using piecewise polynomial interpolants and build the compressed matrix in linear complexity. We call this procedure $\hmatrix(\mathcal T)$ and describe it in Section~\ref{s:hmatrix}.

\renewcommand{\algorithmiccomment}[1]{\bgroup\hfill$\triangleright$~#1\egroup}

The overall algorithm can then be summarized as:
\begin{algorithm}
\caption{Assembly of the stiffness matrix components}\label{alg:assembly}
\begin{algorithmic}
    \STATE{$\underline K = \hmatrix(\mathcal T)$}                              \COMMENT{\emph{Section 5}}
    \STATE{Compute density function $\rho_{\mathcal{T}}$ at all quadrature points}.  \COMMENT{\emph{Section 6.2}}
    \FOR{$\tau\in \mathcal T$ and $\tau\in \mathcal{T}^{\text{int}}$}
    \FOR{$\tau'\subset \mathcal S_\tau$}
    \STATE{Compute $\underline B^{\tau, \tau'} = \texttt{DirectEval}(\tau,\tau')$}    \COMMENT{\emph{Section 4}}
    \STATE{\WL{Assemble $\underline B^{\tau, \tau'}$ into $\underline B$}}
    \ENDFOR
    \STATE{Compute $\underline M^\tau$} the element weighted mass matrix          \COMMENT{\emph{Section 6.1}}
    \STATE{\WL{Assemble $\underline M^\tau$ into $\underline M$}}
    \ENDFOR
    \State \textbf{return}  $\underline B$, $\underline K$, and $\underline M$
\end{algorithmic}  
\end{algorithm}


\section{Direct computation of the singular near-field integrals}
\label{s:direct}

In this section, we describe our $2d$ implementation for computing the singular integrals $\mathcal A_{\tau,\tau'}(\psi_j,\psi_i)$ when two triangles $\tau,\tau'$ are touching each other. The implementation is based on techniques popularized in boundary element methods; see e.g. \cite[Chapter~5]{sauter2010boundary}. We refer to \cite{acosta2017short,ainsworth2017aspects,ainsworth2018towards} for constant-order problems and to \cite{d2021cookbook} for a more general class of nonlocal problems.

Let $\hat \tau$ be the reference triangle with vertices $(0,0)^T$, $(1,0)^T$ and $(1,1)^T$. For each triangle $\tau\in\mathcal T$, we denote by $\chi_\tau: \hat\tau\to \tau$ an affine transformation from the reference triangle to $\tau$. When mapping a touching element pair $(\tau, \tau')$ to $\hat{\tau}$ we distinguish three cases depending on the number of shared vertices (see Figure \ref{fig:duffy}):
\begin{itemize}
    \item If $\tau$ and $\tau'$ share only one vertex, the affine mappings $\chi_\tau$ and $\chi_{\tau'}$ satisfy $\chi_\tau ((0,0)^T) = {\chi_{\tau'}} ((0,0)^T)$;
    \item if $\tau$ and $\tau'$ share a common edge, we assume that $\chi_\tau ((\hat x_1,0)^T) = {\chi_{\tau'}} ((\hat x_1,0)^T)$ for $\hat x_1\in [0,1]$;
    \item if $\tau = \tau'$, we set $\chi_\tau = \chi_{\tau'}$.
\end{itemize}
Under the above assumptions, we set $(x,y) = \chi_{\tau,\tau'}(\hat x,\hat y) := (\chi_\tau(\hat x),\chi_{\tau'}(\hat y))$.

\begin{figure}[ht]
    \vspace*{0pt}
    \begin{center}
        \includegraphics[width=.9\textwidth]{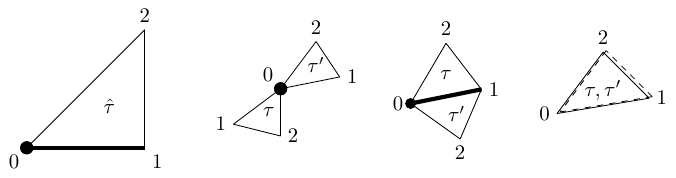}
    \end{center}
    \caption{Mappings $\chi_\tau$, $\chi_{\tau'}$ from reference $\hat{\tau}$ to element pair $(\tau, \tau')$ configurations.}
    \label{fig:duffy}
\end{figure}

Suppose that the supports of basis functions $\psi_i$ and $\psi_j$ contain either $\tau$ or $\tau'$, we shall compute $\mathcal A_{\tau,\tau'}(\psi_j,\psi_i)$ on the reference elements, namely
\begin{equation}\label{e:tgt-ref}
    \begin{aligned}
        \mathcal A_{\tau,\tau'}(\psi_j,\psi_i)
        =  |J_\tau| & |J_{\tau'}|  \int_{\hat\tau}\int_{\hat\tau}
        \bigg( a(\chi_\tau(\hat x),\chi_{\tau'}(\hat y))                                                                                                                                                                                                                                   \\
                    & \times \frac{[\psi_j(\chi_\tau(\hat x)) - \psi_j(\chi_{\tau'}(\hat y))][\psi_i(\chi_\tau(\hat x)) - \psi_i(\chi_{\tau'}(\hat y))]}{|\chi_\tau(\hat x) - \chi_{\tau'}(\hat y)|^{2 + s(\chi_\tau(\hat x)) + s(\chi_{\tau'}(\hat y))}} \bigg)\diff \hat x\diff \hat y ,
    \end{aligned}
\end{equation}
where $J_\tau$ and $J_{\tau'}$ denote the Jacobian of $\chi_\tau$ and $\chi_{\tau'}$ respectively and $|J_\tau|$ and $|J_{\tau'}|$ are the absolute values of the corresponding determinants. The difficulties in the evaluation of the above integral are primarily due to the singularities at $\chi_{\tau}(\hat x) = \chi_{\tau'}(\hat y)$. The strategy for accurate evaluation is to split the integration domain $\hat\tau\times\hat\tau$ into several subregions (depending on the relation between $\tau$ and $\tau'$) so that the integrand in each subregion can be transformed into $[0,1]^4$ and is analytic. We then compute the resulting integrals with tensor-product Gaussian quadrature schemes.

\subsection{Vertex-sharing case}\label{ss:v}
When $\tau$ and $\tau'$ share only one vertex, we use the transformations $(\hat x,\hat y) = \mathfrak T_V^{(i)}(\xi,\boldsymbol \eta)$ with $\xi\in (0,1)$ and $\boldsymbol\eta\in (0,1)^3 \subset \mathbb R^3$, satisfying
\begin{equation}\label{e:map-v}
    \mathfrak T_V^{(1)}(\xi,\boldsymbol\eta) = (\xi,\xi\eta_1,\xi\eta_2,\xi\eta_2\eta_3)
    \quad\text{and}\quad
    \mathfrak T_V^{(2)}(\xi,\boldsymbol\eta) = (\xi\eta_2,\xi\eta_2\eta_3,\xi,\xi\eta_1) .
\end{equation}
So we can decompose $\hat\tau\times\hat\tau$ into two regions and write
\begin{equation}\label{e:int-v}
    \mathcal A_{\tau,\tau'}(\psi_j,\psi_i)
    =  |J_\tau| |J_{\tau'}|\sum_{k=1}^2 \int_{(0,1)^4}a(x,y)
    \frac{(\Phi^{(k)}_{V,j}(\boldsymbol\eta)\Phi^{(k)}_{V,i}(\boldsymbol\eta))\xi^{3-s(x)-s(y)}\eta_2}{|d_V^{(k)}(\boldsymbol\eta)|^{2+s(x)+s(y)}}
    \diff \boldsymbol\eta\diff \xi .
\end{equation}
Here we recall that $(x,y) = \chi_{\tau,\tau'}\circ \mathfrak T_V^{(k)}(\xi,\boldsymbol \eta)$ for $k=1,2$. $\Psi_{V,i}^{(k)}$ provides the difference between $\psi_i(x)-\psi_i(y)$ for all the five associated shape functions and rescaled by $\xi^{-1}$, namely
\[
    \begin{aligned}
        \{\eta_2-1,1-\eta_1,\eta_1,\eta_2\eta_3-\eta_2,-\eta_2\eta_3\}, &  & \quad\text{for } k = 1, \\
        \{1-\eta_2,\eta_2-\eta_2\eta_3,\eta_2\eta_3,\eta_1-1,-\eta_1\}, &  & \quad\text{for } k = 2. \\
    \end{aligned}
\]
The denominator $d_V^{(k)}$ comes from the difference $x-y$ but is rescaled by extracting the factor $\xi$ to give
\begin{equation}\label{e:denomi-v}
    d_V^{(1)}(\boldsymbol\eta) = J_\tau(1,\eta_1)^T - J_{\tau'}(\eta_2,\eta_2\eta_3)^T
    \quad\text{and}\quad
    d_V^{(2)}(\boldsymbol\eta) = J_\tau(\eta_2,\eta_2\eta_3)^T - J_{\tau'}(1,\eta_1)^T .
\end{equation}

The integrand in \eqref{e:int-v} is now non-singular and can thus be approximated by the tensorized Gaussian quadrature rule with order $n$. Here we apply an extra transformation for the variable $\xi$ with $\xi = \zeta^{1/(4-2\overline s)}$ and rewrite \eqref{e:int-v} as
\begin{equation}\label{e:int-v-2}
    \begin{aligned}
        \mathcal A_{\tau,\tau'}(\psi_j,\psi_i)
        =  \frac{|J_\tau| |J_{\tau'}|}{4-2\overline s}\sum_{k=1}^2 \int_{(0,1)^4} & \bigg(a(x,y)\zeta^{\frac{2\overline s-s(x)-s(y)}{4-2\overline s}}\eta_2                                                           \\
                                                                                  & \times\frac{\Phi^{(k)}_{V,j}(\boldsymbol\eta)\Phi^{(k)}_{V,i}(\boldsymbol\eta)}{|d_V^{(k)}(\boldsymbol\eta)|^{2+s(x)+s(y)}}\bigg)
        \diff \boldsymbol\eta\diff \zeta ,
    \end{aligned}
\end{equation}
where  $(x,y) = \chi_{\tau,\tau'}\circ \mathfrak T_V^{(k)}(\zeta^{1/(4-2\overline s)},\boldsymbol \eta)$ . We shall apply the tensorized Gaussian quadrature rule to the above integral and
denote the corresponding approximation of $\mathcal A_{\tau,\tau'}(\psi_j,\psi_i)$ by $Q^n_{V}(\psi_j,\psi_i)$. The reason to compute $\mathcal A_{\tau,\tau'}(\psi_j,\psi_i)$ by \eqref{e:int-v-2} rather than \eqref{e:int-v} is to improve the quadrature approximation rates for a more general setting of $s(x)$ and $a(x,y)$; see Remark~\ref{r:general} below for details.

To analyze the quadrature error, we follow the argument from Section~5.3.2 in \cite{sauter2010boundary}, which is based on the derivative-free quadrature error
estimates \cite{davis1975interpolation}. The error estimate for $\int_{(0,1)^4}f(x_1,\ldots,x_4)\diff \mathbf x$ requires that for each direction $x_i$, the integrand $f$ can be analytically extended from $[0,1]$ to $\mathcal E_{\rho_i}\subset \mathbb C$, where $\mathcal E_{\rho_i}$ is a closed ellipse with the focus points $0$ and $1$, and where $\rho_i>\tfrac12$ denotes the sum of semimajor and semiminor axes. Then, the error between the exact integral and its $n$-th order tensor product Gaussian quadrature, denoted by $Q^n f$, can be estimated by (cf. \cite[Theorem~5.3.15]{sauter2010boundary})
\begin{equation}\label{i:q-err}
    \bigg|\int_{(0,1)^4} f\diff \mathbf x - Q^n f\bigg| \lesssim \sum_{i=1}^4(2\rho_i)^{-2n} \max_{\substack{\mathbf x\in [0,1]^4\\ z\in \partial{\mathcal E}_{\rho_i}}}
    |f(x_1,\ldots,x_{i-1},z,x_{i+1},\ldots, x_4)| .
\end{equation}
To verify that the above estimate can by applied to \eqref{e:int-v}, we shall check that the integrand in \eqref{e:int-v-2} (or \eqref{e:int-v}) can be analytically extended to $\mathcal E_\rho$ for each component with $\rho>\tfrac12$. To this end, we first note that the mappings $\mathfrak T_V^{(k)}$ in \eqref{e:map-v} are component-wise analytic in $\mathbb C$. Thus, the analyticity also holds for the mapping $\chi_{\tau,\tau'}\circ \mathfrak T_V^{(k)}$ as well as the diffusion coefficient $a(x,y)$ and $s(x)+s(y)$ due to the analyticity assumption for $a(.,.)$ and $s(.)$. So $\zeta^{(2\overline s-s(x)-s(y))/(4-2\overline s)}$ (or $\xi^{3-s(x)-s(y)}$) and $|d_V^{(k)}(\boldsymbol\eta)|^{2+s(x)+s(y)}$ are analytic. Noting that the product $(\Phi^{(k)}_{V,j}(\boldsymbol\eta)\Phi^{(k)}_{V,i}(\boldsymbol\eta))$ is a polynomial with degree no more than four, there exists an analytic extension of it in the complex space.

Now we estimate the maximum of the integrand for each component in $(\zeta,\boldsymbol\eta)$. For the numerator, we choose $\rho\in (\tfrac12, 1)$ so that
\begin{equation}\label{i:zeta}
    \sup_{\substack{\zeta\in \mathcal E_\rho\\ \boldsymbol\eta\in[0,1]^3}}|a(x,y)(\Phi^{(k)}_{V,j}(\boldsymbol\eta)\Phi^{(k)}_{V,i}(\boldsymbol\eta))\zeta^{\frac{2\overline s-s(x)-s(y)}{4-2\overline s}}\eta_2| \lesssim 1
\end{equation}
and
\[
    \sup_{\substack{\eta_\ell\in \mathcal E_\rho\\ \eta_k\in[0,1], k\neq \ell\\\zeta\in[0,1]}}|a(x,y)(\Phi^{(k)}_{V,j}(\boldsymbol\eta)\Phi^{(k)}_{V,i}(\boldsymbol\eta))\zeta^{\frac{2\overline s-s(x)-s(y)}{4-2\overline s}}\eta_2| \lesssim 1,\quad\text{for } \ell = 1,2,3 .
\]
For the denominator, we further assume that $h$ is sufficiently small so that $|d_V^{(k)}(\boldsymbol \eta)| \le 1$ (since $|d_V^{(k)}(\boldsymbol \eta)|\sim h$). So there holds
\begin{equation}
    |d_V^{(k)}(\boldsymbol\eta)|^{2+s(x)+s(y)} \ge  |d_V^{(k)}(\boldsymbol\eta)|^{2+2\overline s} .
\end{equation}
This implies that
\[
    \sup_{\substack{\zeta\in \mathcal E_\rho\\ \boldsymbol\eta\in[0,1]^3}}\frac{1}{ |d_V^{(k)}(\boldsymbol\eta)|^{2+s(x)+s(y)} } \lesssim h^{-2-2\overline s}\quad\text{and}\quad
    \sup_{\substack{\eta_\ell\in \mathcal E_\rho\\ \eta_k\in[0,1], k\neq \ell\\\xi\in[0,1]}}\frac{1}{ |d_V^{(k)}(\boldsymbol\eta)|^{2+s(x)+s(y)} } \lesssim h^{-2-2\overline s} .
\]
Gathering the above estimates and invoking \eqref{i:q-err} for the approximation \eqref{e:int-v-2} (or \eqref{e:int-v}) one sees that when $h$ is sufficiently small, there holds
\begin{equation}\label{i:q-v-esti}
    |\mathcal A_{\tau,\tau'}(\psi_j,\psi_i) - Q_{V}^n(\psi_j,\psi_i)|
    \lesssim h^{2-2\overline s} (2\rho)^{-2n} .
\end{equation}

\subsection{Edge-sharing case}\label{ss:e}
When $\tau$ and $\tau'$ share a common edge, we shall use the the following transformations $(\hat  x,\hat y) = \mathfrak T_E^{(k)}(\xi,\boldsymbol \eta)$ for $k=1,\ldots,5$ and $(\xi,\boldsymbol \eta)\in [0,1]^4$:
\[
    \begin{aligned}
        \mathfrak T_E^{(1)}(\xi,\boldsymbol\eta) & = (\xi,\xi\eta_1\eta_3,\xi(1-\eta_1\eta_2), \xi\eta_1(1-\eta_2)) ,               \\
        \mathfrak T_E^{(2)}(\xi,\boldsymbol\eta) & = (\xi,\xi\eta_1,\xi(1-\eta_1\eta_2\eta_3),\xi\eta_1\eta_2(1-\eta_3)),           \\
        \mathfrak T_E^{(3)}(\xi,\boldsymbol\eta) & = (\xi(1-\eta_1\eta_2), \xi\eta_1(1-\eta_2), \xi , \xi\eta_1\eta_2\eta_3) ,      \\
        \mathfrak T_E^{(4)}(\xi,\boldsymbol\eta) & = (\xi(1-\eta_1\eta_2\eta_3), \xi\eta_1\eta_2(1-\eta_3), \xi, \xi\eta_1) ,       \\
        \mathfrak T_E^{(5)}(\xi,\boldsymbol\eta) & = (\xi(1-\eta_1\eta_2\eta_3), \xi\eta_1(1-\eta_2\eta_3), \xi, \xi\eta_1\eta_2) .
    \end{aligned}
\]
We again use the change of variable $\xi = \zeta^{1/(4-2\overline s)}$ to write \eqref{e:tgt-ref} as
\begin{equation}\label{e:int-e}
    \begin{aligned}
        \mathcal A_{\tau,\tau'}(\psi_j,\psi_i)
        = \frac{|J_\tau||J_{\tau'}|}{4-2\overline s}
        \sum_{k=1}^5 \int_{(0,1)^4} & \bigg(a(x,y)\zeta^{\frac{2\overline s-s(x)-s(y)}{4-2\overline s}}J_E^{(k)}(\boldsymbol\eta)                                                                        \\
                                    & \times\frac{\Psi_{E,j}^{(k)}(\boldsymbol\eta)\Psi_{E,i}^{(k)}(\boldsymbol\eta)}{|d_E^{(k)}(\boldsymbol\eta)|^{2+s(x)+s(y)}}\bigg) \diff\boldsymbol\eta\diff\zeta .
    \end{aligned}
\end{equation}
Here $(x,y) = \chi_{\tau,\tau'}\circ \mathfrak T_E^{(k)}(\zeta^{1/(4-2\overline s)},\boldsymbol \eta)$ for $k=1,\ldots,5$. In the numerator above, $J_E^{(1)} = \eta_1^2$ and $J_E^{(k)} = \eta_1^2\eta_2$ are the Jacobians. The functions $\Psi_{E,i}^{(k)}$ are the transformations (rescaled by $\xi^{-1}$) of $\psi_i(x)-\psi_i(y)$ for all four associated shape functions (defined on the corresponding four vertices of $\tau\cup\tau'$):
\[
    \begin{aligned}
        \{-\eta_1\eta_2,\eta_1(1-\eta_3), \eta_1\eta_3, -\eta_1(1-\eta_2)\},             &  & \quad\text{for } k = 1, \\
        \{-\eta_1\eta_2\eta_3,-\eta_1(1-\eta_2),\eta_1,-\eta_1\eta_2(1-\eta_3)\},        &  & \quad\text{for } k = 2, \\
        \{\eta_1\eta_2, -\eta_1(1-\eta_2\eta_3),\eta_1(1-\eta_2),-\eta_1\eta_2\eta_3\},  &  & \quad\text{for } k = 3, \\
        \{\eta_1\eta_2\eta_3, \eta_1(1-\eta_2), \eta_1\eta_2(1-\eta_3),-\eta_1\},        &  & \quad\text{for } k = 4, \\
        \{\eta_1\eta_2\eta_3, -\eta_1(1-\eta_2),\eta_1(1-\eta_2\eta_3), -\eta_1\eta_2\}, &  & \quad\text{for } k = 5. \\
    \end{aligned}
\]
In the denominator, the rescaled distances $d_E^{(k)}(\boldsymbol\eta)$ are
\[
    \begin{aligned}
        d_E^{(1)}(\boldsymbol\eta) & = J_\tau(1,\eta_1\eta_3)^T - J_{\tau'}(1-\eta_1\eta_2,\eta_1(1-\eta_2))^T               \\
        d_E^{(2)}(\boldsymbol\eta) & = J_\tau(1,\eta_1)^T - J_{\tau'}(1-\eta_1\eta_2\eta_3,\eta_1\eta_2(1-\eta_3))^T,        \\
        d_E^{(3)}(\boldsymbol\eta) & = J_\tau(1-\eta_1\eta_2,\eta_1(1-\eta_2))^T - J_{\tau'}(1,\eta_1\eta_2\eta_3)^T         \\
        d_E^{(4)}(\boldsymbol\eta) & = J_\tau(1-\eta_1\eta_2\eta_3, \eta_1\eta_2(1-\eta_3))^T - J_{\tau'}(1,\eta_1)^T,       \\
        d_E^{(5)}(\boldsymbol\eta) & = J_\tau(1-\eta_1\eta_2\eta_3,\eta_1(1-\eta_2\eta_3))^T - J_{\tau'}(1,\eta_1\eta_2)^T .
    \end{aligned}
\]

Following a similar argument as in the previous case, we can show that the function in the above integral can be analytically extended to $\mathcal E_\rho$ with some $\rho\in (\tfrac12,1)$ for each component of $(\xi,\boldsymbol\eta)$. The detailed proof is omitted for brevity. So we can apply the $n$-th order tensorized Gaussian quadrature rule, denoted by $Q_E^n(\psi_j,\psi_i)$, to approximate the integral. Assuming that the mesh size $h$ is small enough so that $d_E^{(k)}(\boldsymbol\eta) < 1$ for $k=1,\ldots, 5$, the quadrature error can be estimated with
\begin{equation}\label{i:q-e-esti}
    |\mathcal A_{\tau,\tau'}(\psi_j,\psi_i) - Q_{E}^n(\psi_j,\psi_i)|
    \lesssim h^{2-2\overline s} (2\rho)^{-2n} .
\end{equation}

\subsection{Identical case}\label{ss:i}
When $\tau=\tau'$ we use the following transformations $(\hat x,\hat y) = \mathfrak T_I^{(k)}(\xi,\boldsymbol\eta)$ for $k=1,\ldots,6$ and $(\xi,\boldsymbol \eta)\in [0,1]^4$:
\[
    \begin{aligned}
        \mathfrak T_I^{(1)}(\xi,\boldsymbol\eta) & = (\xi, \xi(1-\eta_1+\eta_1\eta_2), \xi(1-\eta_1\eta_2\eta_3), \xi(1-\eta_1)) ,            \\
        \mathfrak T_I^{(2)}(\xi,\boldsymbol\eta) & = (\xi, \xi(1-\eta_2+\eta_2\eta_3), \xi(1-\eta_1\eta_2), \xi\eta_1(1-\eta_2)) ,            \\
        \mathfrak T_I^{(3)}(\xi,\boldsymbol\eta) & = (\xi(1-\eta_1\eta_2\eta_3), \xi\eta_1(1-\eta_2\eta_3), \xi, \xi\eta_1(1-\eta_2)) ,       \\
        \mathfrak T_I^{(k)}(\xi,\boldsymbol\eta) & = \mathfrak S\circ \mathfrak T_I^{(k-3)}(\xi,\boldsymbol\eta),\quad\text{for } k = 4,5,6 .
    \end{aligned}
\]
with $\mathfrak{S}(\hat x,\hat y) = (\hat y,\hat x)$. Thanks to the symmetry property between the mappings $\mathfrak T_I^{(k)}$ and $\mathfrak T_I^{(k+3)}$ for $k=1,2,3$ (i.e., the operator $\mathfrak{S}$) as well as the symmetry property for the kernel function, we arrive at
\begin{equation}\label{e:int-i}
    \begin{aligned}
        \mathcal A_{\tau,\tau'}(\psi_j,\psi_i) = \frac{|J_\tau|^2}{2-\overline s}
        \sum_{k=1}^3\int_{(0,1)^4} & \bigg(\xi^{\frac{2\overline s-s(x)-s(y)}{4-2\overline s}}\eta_1^{2-s(x)-s(y)}\eta_2^{1-s(x)-s(y)}                                              \\
                                   & \times\frac{a(x,y)\Psi_{I,j}^{(k)}(\eta_3)\Psi_{I,i}^{(k)}(\eta_3)}{|d_I^{(k)}(\eta_3)|^{2+s(x)+s(y)}}\bigg) \diff\boldsymbol\eta\diff \zeta ,
    \end{aligned}
\end{equation}
where $(x,y) = \chi_{\tau,\tau'}\circ \mathfrak T_I^{(k)}(\xi,\boldsymbol \eta)$ with $\xi = \zeta^{1/(4-2\overline s)}$, $\Psi_{I,i}^{(k)}(\eta_3)$ are the three rescaled shape functions ($\psi_i(x)-\psi_j(y)$) defined on $\tau$ provided by
\[
    \begin{aligned}
        \{-\eta_3,\eta_3-1,1\}, &  & \quad\text{for } k = 1, \\
        \{-1,1-\eta_3,\eta_3\}, &  & \quad\text{for } k = 2, \\
        \{\eta_3,-1,1-\eta_3\}, &  & \quad\text{for } k = 3, \\
    \end{aligned}
\]
and
\[
    d_I^{(1)}(\eta_3) = J_\tau(\eta_3,1)^T,\,
    d_I^{(2)}(\eta_3) = J_\tau(1,\eta_3)^T,\,
    d_I^{(3)}(\eta_3) = J_\tau(\eta_3,1-\eta_3)^T
\]
are the rescaled distances between $x$ and $y$.
Notice that the integrand in \eqref{e:int-i} could be singular at $\eta_2 = 0$ due to the term $\eta_2^{1-s(x)-s(y)}$ when $s(x)+s(y)>1$. To resolve this, we use the change of variable $\eta_2 = t^{1/(2-2\overline s)}$ to write
\begin{equation}\label{e:int-i-2}
    \begin{aligned}
        \mathcal A_{\tau,\tau'}(\psi_j,\psi_i) = \frac{|J_\tau|^2}{(2-\overline s)(2-2\overline s)}
        \sum_{k=1}^3\int_{(0,1)^4} & \bigg(\xi^{\frac{2\overline s-s(x)-s(y)}{4-2\overline s}}\eta_1^{2-s(x)-s(y)}t^{\tfrac{2\overline s-s(x)-s(y)}{2-2\overline s}  }                        \\
                                   & \times\frac{a(x,y)\Psi_{I,j}^{(k)}(\eta_3)\Psi_{I,i}^{(k)}(\eta_3)}{|d_I^{(k)}(\eta_3)|^{2+s(x)+s(y)}}\bigg) \diff \eta_1\diff t\diff\eta_3\diff \zeta ,
    \end{aligned}
\end{equation}
where we recall that $(x,y) = \chi_{\tau,\tau'}\circ \mathfrak T_I^{(k)}(\zeta^{1/(4-2\overline s)},\eta_1, t^{1/(2-2\overline s)},\eta_3)$. We then apply the tensorized Gaussian quadrature to \eqref{e:int-i-2}. If $\overline s<\tfrac12$, we can also simply apply the same quadrature rule to \eqref{e:int-i}. Denoting $Q_I^n(\psi_j,\psi_i)$ the resulting quadrature approximation with order $n$ based on \eqref{e:int-i} or \eqref{e:int-i-2}. Following the arguments above, we obtain that when $h$ is small enough, there holds
\begin{equation}\label{i:q-i-esti}
    |\mathcal A_{\tau,\tau'}(\psi_j,\psi_i) - Q_{I}^n(\psi_j,\psi_i)|
    \lesssim h^{2-2\overline s} (2\rho)^{-2n} ,
\end{equation}
for some $\rho\in (\tfrac12,1]$.

\subsection{Quadrature error for the singular near-field formulation}
The quadrature schemes provided by the proceeding subsections for all $\tau\in \Oi$ and $\tau'\in\mathcal S_\tau$ form an approximation of the bilinear form $\mathcal B(U,V)$ for $U,V\in \mathbb V(\mathcal T)$. We denote this approximation by $Q^n_{\mathcal B}(U,V)$. The following proposition shows the corresponding consistency error. The proof follows the standard arguments for the quadrature approximation for bilinear forms (see e.g. \cite[Theorem~10]{ainsworth2017aspects} and \cite[Theorem~5.3.29]{sauter2010boundary}). Here we provide a proof for completeness.
\begin{proposition}[quadature error for near-field approximations]\label{p:b-q}
    For $U,V\in\mathbb V(\mathcal T)$, let $Q^n_{\mathcal B}(U,V)$ be the approximation of $\mathcal B(U,V)$ by replacing $\mathcal A_{\tau,\tau'}(U,V)$ with $\mathcal Q^n_V$, $\mathcal Q^n_E$ or $\mathcal Q^n_I$ defined in \Cref{ss:v,ss:e,ss:i} depending on the relations between $\tau$ and $\tau'$. When the mesh size $h$ is small enough, there exists $\rho\in(\tfrac12,1)$ such that
    \[
        |\mathcal B(U,V) - \mathcal Q^n_{\mathcal B}(U,V)|
        \lesssim h^{-2\overline s}(2\rho)^{-2n} \|U\|_{L^2(\Oi)}\|V\|_{L^2(\Oi)} .
    \]
\end{proposition}
\begin{proof}
    Denote $e^{ij}_{\tau,\tau'} = \mathcal A_{\tau,\tau'}(\psi_j,\psi_i) - Q^n(\psi_j,\psi_i)$, where $Q^n$ is the tensor product Gaussian quadrature form $Q_V^n$, $Q_E^n$, or $Q_I^n$. According to the quadrature error estimates \cref{i:q-v-esti,i:q-e-esti,i:q-i-esti}, there holds that for $h$ sufficiently small,
    \[
        |e^{ij}_{\tau,\tau'}| \lesssim h^{2-2\overline s}(2\rho)^{-2n} .
    \]
    Now we set $U=\sum_{j=1}^Nu_j\psi_j$ and $V=\sum_{i=1}^N v_i\psi_i$. To estimate the target error, we use the definition \eqref{e:b} and write
    \begin{equation}\label{i:b-form-err}
        \begin{aligned}
            |\mathcal B(U,V) - \mathcal Q^n_{\mathcal B}(U,V)| & \le \sum_{\substack{\tau\in\mathcal T^{\text{int}}                                     \\ \tau'\in\mathcal S_\tau}} \sum_{i,j\in \mathcal I(\tau,\tau')}  |e^{ij}_{\tau,\tau'}u_jv_i|                                    \\
                                                               & \lesssim h^{2-2\overline s}(2\rho)^{-2n}\sum_{\substack{\tau\in\mathcal T^{\text{int}} \\ \tau'\in\mathcal S_\tau}} \sum_{i,j\in \mathcal I(\tau,\tau')} |u_jv_i|.
        \end{aligned}
    \end{equation}
    Here, we recall that $\mathcal I(\tau,\tau')\subset \{1,\ldots, N\}$ is the index set whose associated global shape functions are non-zero on $\tau\cup\tau'$. Utilizing the Cauchy-Schwarz inequality, we have
    \begin{equation}\label{i:sum-1}
        \begin{aligned}
            \sum_{i,j\in\mathcal I(\tau,\tau')} |u_jv_i| & \le \#\mathcal I(\tau,\tau')
            \bigg(\sum_{j\in \mathcal I(\tau,\tau')}|u_j|^2\bigg)^{1/2}
            \bigg(\sum_{i\in \mathcal I(\tau,\tau')}|v_i|^2\bigg)^{1/2}                                                         \\
                                                         & \lesssim h^{-2}\|U\|_{L^2(\tau\cup\tau')}\|V\|_{L^2(\tau\cup\tau')},
        \end{aligned}
    \end{equation}
    where $\#\mathcal I(\tau,\tau')$ denotes the cardinality of $\mathcal I(\tau,\tau')$ and where for the last inequality we used the fact that $\#\mathcal I(\tau,\tau')\lesssim 1$ and that
    \[
        \sum_{j\in \mathcal I(\tau,\tau')}|u_j|^2\lesssim h^{-2}\|U\|_{L^2(\tau\cup\tau')}^2 .
    \]
    We note that due to the shape-regularity property for $\mathcal T$, the number of cells in $\mathcal S_\tau$ is uniformly bounded for all $\tau\in \Oi$. So we insert \eqref{i:sum-1} into \eqref{i:b-form-err} and continue to bound the summation in \eqref{i:b-form-err} by the Cauchy-Schwarz inequality. This leads to
    \[
        \begin{aligned}
            \sum_{\substack{\tau\in\mathcal T^{\text{int}}                                      \\ \tau'\in\mathcal S_\tau}} & \sum_{i,j\in \mathcal I(\tau,\tau')} |u_jv_i|
            \lesssim     h^{-2}\sum_{\tau\in\mathcal T^{\text{int}} }
            \sum_{\tau'\in\mathcal S_\tau} \|U\|_{L^2(\tau\cup\tau')}\|V\|_{L^2(\tau\cup\tau')} \\
             & \lesssim h^{-2}\bigg(\sum_{\substack{\tau\in\mathcal T^{\text{int}}              \\ \tau'\in\mathcal S_\tau}} \|U\|_{L^2(\tau\cup\tau')}^2\bigg)^{1/2}
            \bigg(\sum_{\substack{\tau\in\mathcal T^{\text{int}}                                \\ \tau'\in\mathcal S_\tau}} \|V\|_{L^2(\tau\cup\tau')}^2\bigg)^{1/2} \\
             & \lesssim h^{-2} \|U\|_{L^2(\Oi)}\|V\|_{L^2(\Oi)} .
        \end{aligned}
    \]
    Together with \eqref{i:b-form-err}, we arrive at
    \[
        |\mathcal B(U,V) - \mathcal Q^n_{\mathcal B}(U,V)|
        \lesssim h^{-2\overline s}(2\rho)^{-2n} \|U\|_{L^2(\Oi)}\|V\|_{L^2(\Oi)}
    \]
    as desired.
\end{proof}

We end the section with the following remarks.

\begin{remark}[order of the quadrature rule]\label{r:order}
    In order to get the convergence rate $h^{\beta}$ with $\beta>0$ for the consistency error in Proposition~\ref{p:b-q}, we need to set
    \[
        h^{-2\overline s}(2\rho)^{-2n} \le Ch^{\beta}
    \]
    with some positive constant $C>1$. This implies that the quadrature order should be chosen to satisfy that
    \[
        n\ge \frac{(\beta+2\overline s)\log(1/h)-\log C}{2\log(2\rho)} .
    \]
\end{remark}
\begin{remark}[near-field consistency]
    Note that by the fact that $\tau'\in \mathcal S_\tau$ and that $\mathcal S_\tau$ is quasi-uniform, the error estimate we obtained in Proposition~\ref{p:b-q} improves the results from \cite[Theorem~10]{ainsworth2017aspects} by removing the factor $N$ (the number of degrees of freedom).
\end{remark}
\begin{remark}[a general setting for $s(x)$ and $a(x,y)$]\label{r:general}
    The above implementation for the matrix $\underline B$ can be also extended when $s(x)$ has jumps across the edges. Here we assume that for each triangle $\tau\in\mathcal T$, $s(\chi_\tau(\hat x))$ can be analytically extended to a complex neighborhood of the reference triangle $\hat\tau$, denoted by $\hat\tau^*$. Similar assumptions can be also applied to $a(.,y)$ and $a(x,.)$. Following the argument in \cite[Lemma~5.3.19]{sauter2010boundary}, we can show that $s(\chi_\tau\circ\mathfrak T_V^{(k)}(\zeta,\boldsymbol\eta))$ is analytic for  $\zeta\in\mathcal E_{\rho_1/h}$ with $\rho_1>0$ sufficiently small and is analytic for $\eta_\ell\in \mathcal E_{\rho_2}$ for some $\rho_2\in (\tfrac12, 1)$ with $\ell=1,2,3$. Under the above restriction for $\zeta$, we shall update the estimate \eqref{i:zeta} with
    \[
        \sup_{\substack{\zeta\in \mathcal E_\rho\\ \boldsymbol\eta\in[0,1]^3}}|a(x,y)(\Phi^{(k)}_{V,j}(\boldsymbol\eta)\Phi^{(k)}_{V,i}(\boldsymbol\eta))\zeta^{\frac{2\overline s-s(x)-s(y)}{4-2\overline s}}\eta_2| \lesssim h^{-\frac{\overline s-\underline s}{2-\overline s}} .
    \]
    This implies the new error estimate for approximating \eqref{e:int-v-2}
    \[
        \begin{aligned}
            |\mathcal A_{\tau,\tau'}(\psi_j,\psi_i) - Q_V^n(\psi_j,\psi_i)|
             & \lesssim \left(\frac{\rho_1}h\right)^{-2n}h^{2-2\overline s-\frac{\overline s-\underline s}{2-\overline s}} + h^{2-2\overline s}(2\rho_2)^{-2n} \\
             & \lesssim h^{2n+2-2\overline s-\frac{\overline s-\underline s}{2-\overline s}} +   h^{2-2\overline s}(2\rho_2)^{-2n} .
        \end{aligned}
    \]
    On the other hand, if we approximate \eqref{e:int-v-2}, we have
    \[
        |\mathcal A_{\tau,\tau'}(\psi_j,\psi_i) - Q_V^n(\psi_j,\psi_i)|
        \lesssim h^{2n+2\underline s-2\overline s-1} +   h^{2-2\overline s}(2\rho_2)^{-2n} .
    \]
    Here we note that for the first term on the right-hand side above, the convergence rate is lower than the previous approach and could even be negative when $n=1$.
    We can analogously analyze the quadrature for \cref{e:int-e,e:int-i-2} and follow the argument in Proposition~\ref{p:b-q} to obtain that
    \[
        |\mathcal B(U,V) - \mathcal Q^n_{\mathcal B}(U,V)|  \lesssim h^{2n-1-2\overline s-\frac{\overline s-\underline s}{2-\overline s}} +   h^{-2\overline s}(2\rho_2)^{-2n}.
    \]
    Hence, when choosing the quadrature order according to Remark~\ref{r:order} to achieve the rate $O(h^\beta)$, a sufficient condition is to set
    \[
        n\ge \frac12(\beta+1+2\overline s + \frac{\overline s-\underline s}{2-\overline s}) .
    \]
    In our numerical simulations in Section~\ref{s:numerics}, we set $n=1$ when assembling $\underline B$.
\end{remark}

\section{\texorpdfstring{$\mathcal{H}^2$}{H2}-matrix approximation of the non-singular interactions}
\label{s:hmatrix}

We now consider the construction of a hierarchical matrix approximation of the second term in \eqref{e:decomp} so as to avoid the quadratic complexity that a direct computation would entail. The construction here is now dealing with a de-singularized kernel, since the singularities due to element-pairs that are touching have been resolved by the integrations of the previous section.

\subsection{Construction of the \texorpdfstring{$\mathcal{H}^2$}{H2} matrix structure}
One of the key approximations in hierarchical matrix representations involves clustering neighboring vertices and representing their net effect on other, sufficiently far-away, clusters by appropriate $2$-dimensional polynomials. Therefore the first step of the construction is to generate a hierarchy of spatial clusters for the mesh vertices. We do this by partitioning the interior vertices using a KD-tree, with repeated plane splits along coordinate directions. The construction is recursive starting from the whole point set as the topmost cluster. The points within each cluster are first sorted by projecting along the largest dimension of their bounding box. The sorted point clusters are then split along their median into two children clusters, with the recursion stopping when the cardinality of leaf clusters reaches a specified parameter $m$. This procedure produces a complete binary cluster tree $\mathscr T$ that has $L = 1 + \log(N/m)$ levels with leaves of size no larger than $m$.

The resulting cluster tree $\mathscr T$ together with an admissibility condition provides the structure and the starting point for constructing the $\mathcal H$-matrix approximation $\widetilde{\underline K}$ of $\underline K$.
Specifically, let $\sigma$ and $\sigma'$ be the vertex index sets for two clusters at the same level $\ell$ in $\mathscr T$ and $\omega_{\sigma}$ and $\omega_{\sigma'}$ their corresponding bounding boxes, respectively.
The matrix block $\widetilde{\underline{K}}_{\sigma,\sigma'}$ with rows $\sigma$ and columns $\sigma'$ may be represented as a single low rank approximation if the corresponding bounding boxes satisfy the admissibility condition
\begin{equation}\label{i:admiss}
    \max\{\diam(\omega_{\sigma}^e), \diam(\omega_{\sigma'}^e)\} \le \lambda_1\dist(\omega_{\sigma}^e,\omega_{\sigma'}^e)
\end{equation}
for some $\lambda_1 > 0$.
Here $\omega_{\sigma}^e$ and $\omega_{\sigma'}^e$ are the bounding-box extensions containing all elements in the support of the basis functions of nodes in $\sigma$ and $\sigma'$. This extends the elements in the clusters whose bounding boxes are $\omega_{\sigma}$ and $\omega_{\sigma'}$ by a band that is one-element wide (see Figure \ref{fig:h2a}), and therefore insures that no element pair ($\tau$, $\tau'$) from the two bounding boxes $\omega_{\sigma}^e$ and $\omega_{\sigma'}^e$ satisfying \eqref{i:admiss} involves singular integrals, and that the kernel $\gamma_{\mathcal{T}}(x, y) = \gamma(x, y)$ in this cluster pair. We denote by $\pfar$ the collection of the cluster pairs $(\sigma,\sigma')\in \mathscr T\times\mathscr T$ that satisfy the admissibility condition and describe the computation of their low rank approximation in section \ref{sec:pfar} below.

When the admissibility condition is not satisfied for clusters at the leaf level $L$, the entries in the blocks $\widetilde{\underline K}_{\sigma,\sigma'}$ are computed by direct numerical quadrature. Since we have already accounted for the singular integrals involving element pairs $(\tau, \tau')$ with $\tau \cap \tau' \ne 0$ directly in section 4, the integrals here involve only smooth integrands and we describe their computation in section \ref{sec:pnear} below. We will denote by $\pnear$ the set of leaf-level cluster pairs that violate the admissibility condition and are computed directly. They represent the non-singular smooth part of the near field. Hence, $\pfar\cup\pnear$ covers all the cluster pairs.

The construction of the hierarchical structure of the matrix is recursive and starts from the root of the matrix quadtree $\mathscr T\times\mathscr T$. At every level $l$, the pairs $(\sigma,\sigma')$ with $\sigma$ in the first (row) tree and $\sigma'$ in the second (column) tree are considered. If a pair satisfies the admissibility condition, a low rank approximation of it is constructed from a suitable a polynomial approximation of the kernel, and the corresponding matrix block is no longer subdivided. If the admissibility condition is violated, the children of $\sigma$ and $\sigma'$ at level $l+1$ are considered. The recursion terminates when the leaf level $L$ is reached. The blocks at that level with clusters that do not satisfy the admissibility condition are computed directly as described next; this includes the diagonal blocks with
$\sigma = \sigma'$.
\begin{figure}[ht]
    \vspace*{-4pt}
    \begin{center}
        \begin{subfigure}{0.49\textwidth}
            \begin{center}
                \includegraphics[width=\textwidth]{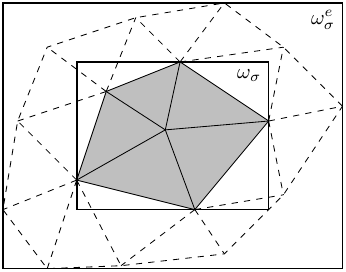}
            \end{center}
            \caption{}
            \label{fig:h2a}
        \end{subfigure}
        \hfill
        \begin{subfigure}{0.39\textwidth}
            \begin{center}
                \includegraphics[width=\textwidth]{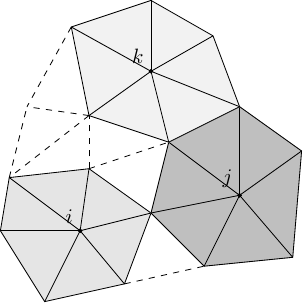}
            \end{center}
            \caption{}
            \label{fig:h2b}
        \end{subfigure}
    \end{center}
    \caption{$\mathcal{H}^2$ computations. (a) $\omega_{\sigma}^e$ extends the bounding box $\omega_{\sigma}$ of the shaded cluster $\sigma$ to include a band that is one-element wide. (b) Near-field entries: entry $A_{ik}$ sums contributions from the 36 elements pairs in $\mathcal{S}_i$ and $\mathcal{S}_k$, but entry $A_{ij}$ ignores the four vertex-touching pairs. }
    \label{fig:h2}
\end{figure}

\subsection{Direct computation of the smooth near-field} \label{sec:pnear}
Let us first consider the (non singular) near-field entries.
For the entry $\widetilde{\underline K}_{ij}$ for $i,j=1,\ldots,N$ with $i\in \sigma$, $j\in\sigma'$, and $(\sigma,\sigma')\in \pnear$, we have
\begin{equation}\label{e:k-near}
    \begin{aligned}
        \underline K_{ij} & = \sum_{\substack{\tau\subset\mathcal S_i, \tau'\subset\mathcal S_j     \\ \tau\cap \tau'=\emptyset}}
        \mathcal K_{\tau,\tau'}(\psi_j,\psi_i)                                                      \\
                          & = -2 \sum_{\substack{\tau\subset\mathcal S_i, \tau'\subset\mathcal S_j  \\ \tau\cap \tau'=\emptyset}} \int_\tau\int_{\tau'}\gamma(x,y)\psi_j(x)\psi_i(y)\diff y\diff x \\
                          & = -2  \sum_{\substack{\tau\subset\mathcal S_i, \tau'\subset\mathcal S_j \\ \tau\cap \tau'=\emptyset}}|J_\tau||J_{\tau'}|\int_{\hat\tau}\int_{\hat\tau} \gamma(\chi_\tau(\hat x),\chi_{\tau'}(\hat y))\hat\psi_j(\hat x)\hat\psi_i(\hat y) \diff \hat y\diff \hat x ,
    \end{aligned}
\end{equation}
where $\hat\psi_i$ are the corresponding shape functions defined on $\hat\tau$.
The sum is over non-touching element pairs in the support of $\psi_i$ and $\psi_j$ (see Figure \ref{fig:h2b}).
The computation for every such element pair (the integral on the right-hand side above) is based on a strategy similar to the one introduced in Section~\ref{s:direct}. Here we apply the Duffy coordinates $\mathfrak T : [0,1]^4\to \hat\tau\times\hat\tau$ satisfying that
\[
    \mathfrak T(\xi,\boldsymbol\eta) = (\xi,\xi\eta_1,\eta_2,\eta_2\eta_3) = (\hat x,\hat y).
\]
We then again apply the tensor-product Gaussian quadrature scheme with order $n$ to the transformed integral and denote by $\widetilde{\underline K}_{ij}$ the resulting approximation of $\underline K_{ij}$.

\subsection{Approximation of the far-field}
\label{sec:pfar}
Let's consider a matrix block $\widetilde{\underline K}_{\sigma,\sigma'}$ with $(\sigma,\sigma')\in\pfar$, which is represented as a low rank block. The $ij$-th entry of that block, $(\widetilde{\underline K}_{\sigma,\sigma'})_{ij}$ for $i\in \sigma$ and $j\in\sigma'$ may be written as:
\begin{equation}\label{e:k-far}
    \begin{aligned}
        (\widetilde{\underline K}_{\sigma,\sigma'})_{ij} & = \sum_{\substack{\tau\subset\mathcal S_i,\tau'\subset\mathcal S_j  \\ \tau\cap\tau'=\emptyset}}\widetilde{\mathcal K}_{\tau,\tau'}(\psi_j,\psi_i)\\
                                                         & := \sum_{\substack{\tau\subset\mathcal S_i,\tau'\subset\mathcal S_j \\ \tau\cap\tau'=\emptyset}}
        -2\int_\tau\int_{\tau'} I^p_{\sigma,\sigma'}\gamma(x,y)\psi_i(x)\psi_j(y)\diff y \diff x ,
    \end{aligned}
\end{equation}
where $I^{p}_{\sigma,\sigma'}$ is the nodal interpolant on the domain $\omega_{\sigma}^e\times\omega_{\sigma'}^e$ using a tensor-product Chebyshev polynomial of order $p$,
\begin{equation} \label{eq:interpolant}
    I^{p}_{\sigma,\sigma'}\gamma(x, y) = \sum_{\alpha=1}^p \sum_{\beta=1}^p \gamma(\xi_{\sigma, \alpha}, \xi_{\sigma', \beta}) \ell_{\sigma, \alpha}(x) \ell_{\sigma', \beta}(y) \quad \text{for } x \in \omega^e_\sigma, y \in \omega^e_{\sigma'}
\end{equation}
where $\{\xi_{\sigma,\alpha},\xi_{\sigma',\beta}\}$ are Chebyshev nodes and $\{\ell_{\sigma,\alpha}, \ell_{\sigma',\beta}\}$ are the corresponding $d$-dimensional Lagrange basis polynomials in $\omega^e_\sigma$ and $\omega^e_{\sigma'}$, respectively. This allows us to write the $ij$-th entry of $\widetilde{\underline K}_{\sigma,\sigma'}$ as
\begin{equation} \label{eq:triple_product}
    \sum_{\alpha=1}^p \sum_{\beta=1}^p
    \left[ \sum_{\substack{\tau\subset\mathcal S_i}} \int_\tau \psi_i(x) \ell_{\sigma, \alpha} (x) dx \right]
    \left[ \vphantom{\int} \! -2 \gamma(\xi_{\sigma, \alpha}, \xi_{\sigma', \beta}) \right]
    \left[ \sum_{\substack{\tau'\subset\mathcal S_j}} \int_{\tau'} \psi_j(y) \ell_{\sigma', \beta} (y) dy \right]
\end{equation}
which in factored form is given by
\begin{equation} \label{eq:triple_product_factored}
    (\widetilde{\underline K}_{\sigma,\sigma'})_{ij} =
    \sum_{\alpha=1}^p \sum_{\beta=1}^p
    \left( \underline{U}_\sigma \right)_{i\alpha} \left(\underline{S}_{\sigma, \sigma'} \right)_{\alpha \beta} \left(\underline{V}_{\sigma'}^T \right)_{\beta j} =
    \left( \underline{U}_\sigma \, \underline{S}_{\sigma, \sigma'} \, \underline{V}_{\sigma'}^T \right)_{ij}.
\end{equation}

The representation $\underline{U}_\sigma \, \underline{S}_{\sigma, \sigma'} \, \underline{V}_{\sigma'}^T$
is a rank-$p$ factorization of the $\widetilde{\underline K}_{\sigma,\sigma'}$ block. It is written in the bases $\underline{U}_\sigma$ and $\underline{V}_{\sigma'}$. These bases are of size ($\#\sigma \times p$)  and ($\#\sigma' \times p$) and
are common to all block rows $\sigma$ and block columns $\sigma'$ in $\widetilde{\underline K}$, respectively. Individual matrix blocks have their own small $\underline{S}_{\sigma, \sigma'}$ ($p \times p$) factors. Evaluation of the bases
can be done exactly using $\lceil\tfrac{p+1}2\rceil$-order Gaussian quadrature schemes since the integrands involved in their entries
are polynomials of order no more than $p+1$.

There is one final step needed to achieve linear overall complexity, since the approximation of the admissible blocks by the low rank approximation  $\underline{U}_\sigma \, \underline{S}_{\sigma, \sigma'} \, \underline{V}_{\sigma'}^T$
above would result in $O(N \log N)$ complexity. In order to remove the $\log N$ factor we can build \emph{nested} bases to avoid generating and storing $\underline{U}_\sigma$ and $\underline{V}_{\sigma'}$  explicitly for all levels of the hierarchy. This can be done by expressing the $p$ polynomial bases used in the approximation over a region $\omega^e_\sigma$ at level $l$ in terms of the approximating polynomials over the subregions of its children clusters in the cluster tree $\mathscr T$. In practice, this allows us to generate and store the bases $\underline{U}_\sigma$ and $\underline{V}_{\sigma'}$ explicitly at the leaf level only, with small inter-level transfer matrices that allow the implicit generation of the bases at coarser levels, recursively. This hierarchical (nested) basis of the hierarchically partitioned matrix is called the $\mathcal{H}^2$ representation and attains the optimal complexity \cite{borm2010efficient}. While asymptotically optimal, the thus constructed matrix does not generally have optimal constants, as it uses a generic polynomial basis for the construction. As a result, we are able to further compress the matrix algebraically and reduce the ranks of the matrix blocks and the overall memory footprint of the matrix. The details of this algebraic compression process and a demonstration of its effectiveness are described in \cite{boukaram2020hierarchical,H2Opus}.

\subsection{Consistency}

There are two approximation errors that need to be analyzed:
(i) the local quadrature error for $\mathcal K_{\tau,\tau'}(\psi_j,\psi_i)$ when $(\sigma,\sigma')\in \pnear$,
and (ii) the local interpolation error for $I_{\sigma,\sigma'}^p$ when $(\sigma,\sigma')\in \pfar$.
For $U,V\in \mathbb V(\mathcal T)$, we denote by $\widetilde{\mathcal K}^{\text{near}}(U,V)$ the bilinear form associated with the near-field part, i.e., the non-zero entries in \eqref{e:k-near}, and $\widetilde{\mathcal K}^{\text{far}}(U,V)$ the far-field part defined in \eqref{e:k-far}. So $\mathcal K(U,V)$ is approximated by
\[
    \mathcal K(U,V) \approx \widetilde K(U,V):= \widetilde{\mathcal K}^{\text{near}}(U,V) +   \widetilde{\mathcal K}^{\text{far}}(U,V) .
\]
We can also similarly decompose $\mathcal K(.,.)$ to the near-field part $\mathcal K^{\text{near}}(.,.)$ and the far-field part $\mathcal K^{\text{far}}(.,.)$ and we shall estimate their errors separately.

\subsubsection{Near-field}
In order to bound the error from the near-field part, we first note that the cardinality of $\pnear$ is uniformly bounded. Hence, the near-field entries in $\widetilde{\underline K}$ form a sparse matrix. In order to show this, we first note that thanks to the quasi-uniformity assumption on $\mathcal T$, for each leaf cluster $\sigma\in \mathscr T$, $\omega_\sigma$ is regular and satisfies that $|\omega_\sigma|\sim 2^{-l_{\max}}|\Oi|$. For $(\sigma,\sigma')\in\pnear$ and $x\in \omega_{\sigma}^e$ and $y\in\omega_{\sigma'}^e$, there holds
\begin{equation}\label{i:near-r}
    \begin{aligned}
        |x-y| & \le \dist(\omega_{\sigma}^e,\omega_{\sigma'}^e) + \diam(\omega_{\sigma}^e) + \diam(\omega_{\sigma'}^e) \\
              & < \bigg(2+\frac1{\lambda_1}\bigg)\max\{\diam(\omega_{\sigma}^e), \diam(\omega_{\sigma'}^e)\}           \\
              & \le C_1(2^{-l_{\max}/2} + h) =: r,
    \end{aligned}
\end{equation}
where for the second inequality above we used the fact that $(\sigma,\sigma')$ is non-admissible and where for the last inequality we applied the setting for the partition of the cluster tree so that for each leaf $\sigma\in \mathscr T$ there holds
\begin{equation}\label{i:near-vol}
    \diam(\omega_{\sigma}^e) \sim \diam(\omega_{\sigma}) + 2h \sim |\omega_{\sigma}|^{1/2} + 2h
    \sim (2^{-l_{\max}} |\Oi|)^{1/2} + 2h .
\end{equation}
The estimate \eqref{i:near-r} implies that given an index $i$ and a cluster $\sigma$ so that $\mathcal S_i\subset\omega_{\sigma}^e$, the union of the near-field leaves $\sigma$, namely $(\sigma,\sigma')\in\pnear$, is covered by $\Sigma = \cup_{x\in \omega_{\sigma}^e}B_r(x)$. Since $|\Sigma|\lesssim r^2 + |\omega_{\sigma}^e|\lesssim 2^{-l_{\max}} + h^2$, we then utilize \eqref{i:near-vol} to derive that
\[
    \#\{\sigma' : (\sigma,\sigma')\in \pnear\} \lesssim \frac{|\Sigma|}{|\omega_{\sigma}^e|}\lesssim \frac{2^{-l_{\max}} + h^2}{|\omega_{\sigma}^e|} \lesssim 1,
\]
where we note that the above hidden constant depends on $|\Oi|$, $\lambda_1$ as well as the quasi-uniformity constant.

Let $Q^n_K(\psi_j,\psi_i)$ denote the approximation of $K_{\tau,\tau'}(\psi_j,\psi_i)$ in Section~\ref{sec:pnear}. Following the argument from Section~\ref{ss:v} (see also \cite[Lemma~5.3.20 \& Theorem~5.3.24]{sauter2010boundary}), we have that when $h$ is sufficiently small,
\begin{equation}\label{i:err-quad}
    |\mathcal K_{\tau,\tau'}(\psi_j,\psi_i)-Q_{K}^n(\psi_j,\psi_i)| \lesssim h^{2-2\overline s}(2\rho)^{-2n},
\end{equation}
for some $\rho\in (\tfrac12, 1)$. Using the argument from Proposition~\ref{p:b-q} and applying the local error estimate \eqref{i:err-quad}, we can show the consistency of the near-field part by
\[
    |\mathcal K^{\text{near}}(U,V)-\widetilde{\mathcal K}^{\text{near}}(U,V)| \lesssim h^{-2\overline s}(2\rho)^{-2n} \|U\|_{L^2(\Oi)} \|V\|_{L^2(\Oi)}  .
\]

\subsubsection{Far-field}
For the consistency error from the far-field part, we first note that the error estimate for the kernel interpolation $I^p_{\sigma,\sigma'}$  \cite[Theorem~4.22 \& Remark~4.23]{borm2010efficient} (see also \cite[Lemma~5.1]{hackbusch2002h2})
\begin{equation}\label{i:err-interp}
    \begin{aligned}
        \|\gamma(x,y) - I^p_{\sigma,\sigma'}\gamma(x,y)\|_{L^\infty(\tau\times\tau')}
         & \lesssim c_1^p \|\gamma(x,y)\|_{L^\infty(\omega_{\sigma}^e\times\omega_{\sigma'}^e)}  \\
         & \lesssim \frac{c_1^p}{\dist(\omega_{\sigma}^e,\omega_{\sigma'}^e)^{2+2\overline s}} .
    \end{aligned}
\end{equation}
where $c_1=\min\{\tfrac{c_0\lambda_1}{c_0\lambda_1+1},\tfrac{c_0\lambda_1}{2}\}$ for some $c_0>0$ and where the hidden constant depends only on $p$ and $d$. We follow the proof of \cite[Theorem~7.3.18]{sauter2010boundary} to obtain that
\begin{equation}\label{i:err-far}
    |\mathcal K^{\text{far}}(U,V)-\widetilde{\mathcal K}^{\text{far}}(U,V)| \lesssim c_1^p h^{-2-2\overline s} \|U\|_{L^2(\Oi)} \|V\|_{L^2(\Oi)} .
\end{equation}

\subsubsection{Overall error}
Gathering \eqref{i:err-quad} and \eqref{i:err-far} gives the consistency error for the $\mathcal H$-matrix approximation.
\begin{proposition}[consistency for the $\mathcal H^2$-approximation]\label{p:k}
    For $U,V\in \mathbb V(\mathcal T)$, let $\widetilde K(U,V)$ be the approximation of $\mathcal K(U,V)$ by the $\mathcal H^2$-approximation using Chebyshev polynomials with order $p$ together with the admissibility condition \eqref{i:admiss}. A tensor-product Gaussian quadrature rule with order $n$ is used to compute the near-field part of the $\mathcal H$-matrix; see \eqref{e:k-far}. Then there exist constants $\rho\in (\tfrac12, 1]$ and $c_1\in (0,1)$ so that
    \[
        |\mathcal K(U,V)-\widetilde{\mathcal K}(U,V)| \lesssim h^{-2\overline s} (c_1^p h^{-2} + (2\rho)^{-2n}) \|U\|_{L^2(\Oi)} \|V\|_{L^2(\Oi)} .
    \]
\end{proposition}

\begin{remark}[general settings for $s(x)$ and $a(x,y)$]
    The consistency result above relies on the analyticity of $s(x)$. For the case where there are jumps across element edges, the strategy for showing exponential convergence for the $\mathcal H^2$-approximation versus polynomial degree will no longer hold since the $I_{\sigma,\sigma'}^p$ is not well-defined. A more refined argument, following the approach in \cite[Section 9.2 and Theorem 9.5]{borm2010efficient}, is needed to establish that the low rank approximation indeed allows the operator and solution errors to retain an exponential convergence with the rank/degree. The technical details are, however, beyond the scope of the present work.
\end{remark}


\section{Fast multipole acceleration for computing the weighted mass matrix}
\label{s:fmm}

We finally consider the construction of the matrix approximation $\underline M$ of the third term in \eqref{e:decomp}. $\underline M$ has  the footprint of a mass matrix but is weighted by a global density function whose direct computation would require an expensive $O(N^2)$ computation.

\subsection{Element computations}
For each cell $\tau\in\mathcal T^{\text{int}}$, we shall first compute the local contributions $\underline M^\tau$ defined in \eqref{e:mass-l} by the $n$-th order tensor-product quadrature scheme used in \eqref{e:k-near} but only for $x$. We use this scheme primarily for the convenience of having the same code and analysis as the previous sections.
Letting $i,j=1,2,3$ be the indices of the local shape functions in $\tau$, the quadrature scheme leads to
\begin{equation}\label{e:tau-quad}
    \underline M^\tau(\tau)_{i,j} = \int_{\tau}\psi_j(x)\psi_i(x)\rho_{\mathcal T}(x)\diff x
    \approx \sum_{\ell=1}^{N_\tau} \psi_j(q_\ell)\psi_i(q_\ell)\rho_{\mathcal T}(q_\ell) w_\ell
\end{equation}
with $N_\tau$ denoting the number of quadrature points and $\mathcal Q_\tau := \{q_\ell\}_{\ell=1}^{N_\tau}$ and $\mathcal W_\tau := \{w_\ell\}_{\ell=1}^{N_\tau}$ are the quadrature points and weights. Denote the collection of all the quadrature points for $\tau\in\mathcal T^{\text{int}}$ by $\mathcal Q^{\text{int}}$, \ie
\[
    \mathcal Q^{\text{int}} := \bigcup_{\tau\in\mathcal T^{\text{int}}} \mathcal Q_\tau .
\]
We similarly define $\mathcal W^{\text{int}}$ for the quadrature weights. We first show that evaluation of $\rho_{\mathcal T}(q)$ for every $q\in \mathcal Q^{\text{int}}$ requires $O(N)$ operations. Recalling the definition of $\rho_{\mathcal T}(x)$ from \eqref{e:m}, we write the computation for $q_i\in \mathcal Q^{\text{int}}$ as
\[
    \rho_{\mathcal T}(q_i) = \sum_{\tau'\in \mathcal T}
    \int_{\tau'} \gamma_{\mathcal T}(q_i,y) \diff y .
\]
We shall again approximate the right-hand side above by quadrature. Denoting $\mathcal Q^{\text{ext}}$ and $\mathcal W^{\text{ext}}$ the set of quadrature points and weights for $\int_{\Oe}\gamma(q_i,y)\diff y$, we set $\mathcal Q = \mathcal Q^{\text{int}}\cup \mathcal Q^{\text{ext}}$ and $\mathcal W = \mathcal W^{\text{int}}\cup \mathcal W^{\text{ext}}$. Using the quadrature scheme generated by $\mathcal Q$ and $\mathcal W$, we have
\begin{equation}\label{e:tgt-sum}
    \rho_{\mathcal T}(q_i) \approx \sum_{q_j\in \mathcal Q}\gamma_{\mathcal T}(q_i,q_j)  w_j .
\end{equation}
Assuming a suitable subdivision of $\mathcal T$ so that the cardinality $\#(\mathcal Q)$ is $O(N)$, the computation of the right-hand side above obviously requires $O(N)$ operations.

In principle, we can use the kernel-independent fast multipole method \cite{ying2004kernel} to accelerate the evaluations of $\rho_{\mathcal T}(q_i)$ for all $q_i\in \mathcal Q^{\text{int}}$. Here we consider \eqref{e:tgt-sum} as a $N$-body problem by treating $\mathcal Q$, $\mathcal Q^{\text{int}}$ and $\mathcal W$ as source points, target points and source densities, respectively.
For the numerical simulation, one could use available fast multipole open source libraries such as \texttt{exafmm} \cite{wang2021exafmm}, \texttt{PVFMM} \cite{malhotra2015pvfmm}, or \texttt{PBBFMM3D} \cite{darve21}. Unfortunately, these libraries only support a kernel with constant order and constant diffusion coefficients. Our alternative solution is to interpret the general fast multipole method as a hierarchical matrix-vector product in the $\mathcal H^2$ format \cite{yokota14} and again use an $\mathcal H^2$-approximation as we describe below.

\subsection{An \texorpdfstring{$\mathcal H^2$}{H2}-approximation for density evaluation}\label{ss:fmm-h}
We introduce a collocation approach using an $\mathcal H^2$-matrix to compute $\rho_{\mathcal T}(q_i)$ for all $q_i\in\mathcal Q^{\text{int}}$. To this end, denote the space $\mathbb V_s$ the span of Dirac delta distributions for the source points $\mathcal Q$, namely
\[
    \mathbb V_s := \text{span}\{\delta_q : q\in \mathcal Q\}  .
\]
Here $\delta_q$ denotes the Dirac delta distribution at $q$. We similarly define the space $\mathbb V_t$ for the target points $\mathcal Q^{\text{int}}$. We first consider the following rectangular matrix
\[
    (\underline K_\rho)_{ij} = \gamma_{\mathcal T}(q_i,q_j) = \int_\Omega\int_\Omega \gamma_{\mathcal T}(x,y)\delta_{q_i}(x)\delta_{q_j}(y)\diff y\diff x ,\quad \text{for } q_i\in\mathcal Q^{\text{int}} \text{ and }  q_j\in \mathcal Q .
\]
Letting $\underline\rho =  (\rho_{\mathcal T}(q_i))_{q_i\in \mathcal Q^{\text{int}}}^T$ and $\underline w:= (w_j)_{q_j\in \mathcal Q}^T$, we have
\begin{equation}\label{eq:matvec}
    \underline\rho = \underline K_\rho \underline w .
\end{equation}

Therefore, in order to generate the necessary density values at all quadrature points we need to generate an $\mathcal H^2$-approximation of $\underline K_\rho$ and perform the multiplication in \eqref{eq:matvec} efficiently.  The $\mathcal H^2$-approximation algorithm for $\underline K_\rho$ is similar to the one presented in Section~\ref{s:hmatrix} for $\underline K$ and starts by constructing two cluster trees $\mathscr T_t$ and $\mathscr T_s$ for the row index set for $\mathbb V_t$ and the column index set for $\mathbb V_s$, respectively. When building the $\mathcal H^2$-matrix, we do not need to extend the bounding boxes
for the clusters $\sigma \in \mathscr T_t$ and $\sigma' \in \mathscr T_s$; consequently, we may define the interpolation operator $I_{\sigma,\sigma'}^p$ for $\gamma(x,y)$ in $\omega_{\sigma}\times\omega_{\sigma'}$.
The admissibility condition is given by
\begin{equation}\label{i:addmi-q}
    \max\{\diam(\omega_{\sigma}), \diam(\omega_{\sigma'})\} \le \lambda_2\dist(\omega_{\sigma},\omega_{\sigma'}) .
\end{equation}
for some fixed $\lambda_2 >0$. We similarly define $\pfar$ to be the collection of all the admissible blocks $(\sigma,\sigma')$ and define $\pnear$ to be the rest of the blocks. We shall further assume that
$\lambda_2$ is small enough to guarantee that $\dist(\omega_{\sigma},\omega_{\sigma'}) \ge 3h $ so that $\gamma_{\mathcal T}(x,y) = \gamma(x,y)$, and is therefore smooth. Hence, the interpolation $I_{\sigma,\sigma'}^p \gamma_{\mathcal T}$ makes sense for $(\sigma,\sigma')\in \pfar$ and satisfies that
\begin{equation}\label{i:err-h-quad}
    \|\gamma_{\mathcal T}(x,y) - I^p_{\sigma,\sigma'}\gamma_{\mathcal T}(x,y)\|_{L^\infty(\omega_{\sigma}\times\omega_{\sigma'})}
    \lesssim c_2^p \|\gamma(x,y)\|_{L^\infty(\omega_{\sigma}\times\omega_{\sigma'})} \lesssim c_2^p h^{-2-2\overline s},
\end{equation}
where $c_2 = \min\{\tfrac{c_0\lambda_2}{c_0\lambda_2+1},\tfrac{c_0\lambda_2}{2}\}$.

Once the matrix $\underline K_\rho$ is constructed, the matrix-vector multiplication in \eqref{eq:matvec} can be performed via standard multilevel methods for $\mathcal{H}^2$ matrices which involve a pair of upward and downward passes over the basis trees and multiplication by the small low rank blocks at all levels of the hierarchy. The operation can be done in $O(N)$ (cf. \cite{boukaram19a}).

\subsection{Consistency}
Given $U,V\in \mathbb V(\mathcal T)$, let us first denote  the quadrature approximation of $\mathcal M(U,V)$ by $\mathcal Q_{\mathcal M}^n(U,V)$. We also denote by $\mathcal Q_{\mathcal M}^{n,p}(U,V)$ the resulting bilinear form when $\{\rho_{\mathcal T}(q_i)\}_{q_i\in \mathcal Q^{\text{int}}}$ are approximated using the $\mathcal H^2$ matrix-vector product. Here we recall that $p$ is the degree of the Chebyshev polynomials. By the triangle inequality, the consistency error between $\mathcal M(U,V)$ and $\mathcal Q_{\mathcal M}^{n,p}(U,V)$ can be bounded with
\begin{equation}\label{i:tri-quad}
    |\mathcal M(U,V) - \mathcal Q_{\mathcal M}^{n,p}(U,V)|
    \le |\mathcal M(U,V) - \mathcal Q_{\mathcal M}^{n}(U,V)|
    +|\mathcal Q_{\mathcal M}^{n}(U,V) - \mathcal Q_{\mathcal M}^{n,p}(U,V)|
\end{equation}

We first estimate the quadrature error, namely the first error on the right-hand side above. Set $e^{ij}_\tau$ to be the error of the quadrature approximation in \eqref{e:tau-quad}, following the argument in Section~\ref{ss:v}, we have that there exists a constant $\rho\in (\tfrac12, 1]$ so that
\[
    |e^{ij}_\tau|\lesssim (2\rho)^{-2n} h^2\|\rho_{\mathcal T}\|_{L^\infty(\mathbb R^2)}
    \lesssim (2\rho)^{-2n} h^{2-2\overline s} .
\]
Here we note for the last inequality we used the fact that
\[
    \|\rho_{\mathcal T}\|_{L^\infty(\mathbb R^2)}
    \lesssim \int_{ B_{1}(0)\backslash B_{\varepsilon h}(0)} \frac{1}{|y|^{2+2\overline s}}\diff y
    + \int_{\mathbb R^2\backslash B_{1}(0)} \frac{1}{|y|^{2+2\underline s}}\diff y
    \lesssim h^{-2\overline s} ,
\]
where $B_{r}(0)$ is a ball entered at origin with radius $r$ and where $\varepsilon>0$ is sufficiently small. Using the above local error estimate, we again follow the same argument in Proposition~\ref{p:b-q} to derive that
\begin{equation}\label{i:quad}
    |\mathcal M(U,V) - \mathcal Q_{\mathcal M}^{n}(U,V)| \lesssim  (2\rho)^{-2n} h^{-2\overline s} .
\end{equation}

In order to estimate the error from the $\mathcal H^2$-approximation for $\rho_{\mathcal T}$, we let $\sigma\in \mathscr T_t$ and $\sigma'\in \mathscr T_s$.
For $q_i\in \mathcal Q^{\text{int}}$, let $\widetilde \rho_{\mathcal T}(q_i)$ be the resulting approximation of $\rho_{\mathcal T}(q_i)$. We invoke the interpolation error estimate \eqref{i:err-h-quad} as well as $w_j\sim h^2$ to bound the error
\[
    \begin{aligned}
        |\widetilde \rho_{\mathcal T}(q_i) - \rho_{\mathcal T}(q_i)|
         & \le \sum_{\substack{\sigma\in\mathscr T_t                   \\ i\in \sigma}}\sum_{\substack{\sigma'\in\mathscr T_s\\(\sigma,\sigma')\in P^{\text{far}}}} \sum_{j\in \sigma'}|\gamma_{\mathcal T}(q_i,q_j) -  I^p_{\sigma,\sigma'}\gamma_{\mathcal T}(q_i,q_j)|w_j                      \\
         & \lesssim h^2 c_2^{p}  \sum_{\substack{\sigma\in\mathscr T_t \\ i\in \sigma}}\sum_{\substack{\sigma'\in\mathscr T_s\\(\sigma,\sigma')\in P^{\text{far}}}} \sum_{j\in \sigma'}\|\gamma_{\mathcal T}\|_{L^\infty(\omega_{\sigma}\times\omega_{\sigma'})} \lesssim c_2^p h^{-2-2\overline s} ,
    \end{aligned}
\]
where for the last inequality we used the fact the number of far-field indices $j$ is bounded by $N\sim h^{-2}$.
This means that the $\mathcal H^2$-approximation for $\rho_{\mathcal T}(q_\ell)$ leads to the quadrature formula in \eqref{e:tau-quad} perturbed by the error $Cc_2^p h^{-2\overline s}$. Thus we again apply the argument in Proposition~\ref{p:b-q} to obtain that
\begin{equation}\label{i:quad-h}
    |\mathcal Q_{\mathcal M}^{n}(U,V) - \mathcal Q_{\mathcal M}^{n,p}(U,V)| \lesssim  c_2^p h^{-2-2\overline s} \|U\|_{L^2(\Oi)} \|V\|_{L^2(\Oi)} .
\end{equation}
Gathering the errors \eqref{i:quad} and \eqref{i:quad-h} into \eqref{i:tri-quad}, we conclude that
\begin{proposition}[consistency for $\mathcal M(U,V)$]\label{p:m}
    For $U,V\in\mathbb V(\mathcal T)$, let $\mathcal Q_{\mathcal M}^{n,p}(U,V)$ be the resulting approximation of $\mathcal M(U,V)$ by quadrature with order $n$ as well as $\mathcal H^2$-approximation for the quadrature points. Then there exists a constant $\rho\in (\tfrac12,1]$ and $c_2\in (0,1)$ so that
    \[
        |\mathcal M(U,V) - \mathcal Q_{\mathcal M}^{n,p}(U,V)| \lesssim  ((2\rho)^{-2n} + c_2^p h^{-2}) h^{-2\overline s} \|U\|_{L^2(\Oi)} \|V\|_{L^2(\Oi)} .
    \]
\end{proposition}

\subsection{Overall consistency}
We conclude this section with the following theorem by combining the consistency error estimates from \Cref{p:b-q,p:k,p:m}. For simplicity, we will use the same $n$-th order tensor-product Gaussian quadrature to approximate the integral and
the same polynomial degree $p$ for the $\mathcal H$-matrices in \Cref{s:hmatrix,s:fmm}.
\begin{theorem}[total error]\label{t:const}
    For $U,V\in\mathbb V(\mathcal T)$, define the final approximation of $\mathcal A(U,V)$ by
    \[
        \mathcal A_h(U,V) := \mathcal Q_{\mathcal B}^n(U,V) +\widetilde K(U,V) +  \mathcal Q_{\mathcal M}^{n,p}(U,V)
    \]
    where the bilinear forms on right-hand side are defined in \Cref{p:b-q,p:k,p:m}, respectively. Then there holds that
    \[
        |\mathcal A(U,V) - \mathcal A_h(U,V)| \lesssim h^{-2\overline s}((2\rho)^{-2n} + h^{-2} c_1^p + h^{-2} c_2^p)   \|U\|_{L^2(\Oi)} \|V\|_{L^2(\Oi)} .
    \]
\end{theorem}

\section{Numerical illustrations}
\label{s:numerics}

In this section, we present numerical examples to illustrate the performance of our proposed finite element algorithm. In particular, we  report the decay of $L^2(\Oi)$ errors with respect to a sequence of the quasi-uniform meshes as the mesh size $h$ is systematically reduced, and the increase in computational cost as $N$ increases to verify the linear complexity of the algorithm. Our numerical implementation is based on the \texttt{Deal.II} (version~9.4) finite element library \cite{deal.II} which supports simplex meshes and the \texttt{H2Opus} library \cite{H2Opus} for hierarchical matrices. We use the \texttt{TimerOutput} class in \texttt{deal.II} to record the computation time.
Solutions are obtained by a conjugate gradient solver. No attempt was made to fine tune algorithmic parameters, nor to parallelize or optimize the code, which was executed on a single core of a standard-issue laptop computer.

In the construction of the hierarchical matrix approximations for constructing $\underline K$ and $\underline M$, we use the slightly more convenient geometric admissibility condition $\lambda \| C_\sigma - C_{\sigma'} \| \ge (D_{\sigma} + D_{\sigma'})/2$ where $C$ and $D$ refer to the center and diameter of the bounding box ($\omega$ or $\omega^e$) used for a cluster. We use $\lambda=0.75$, leaf size $m=128$, and approximate the kernel function $\gamma_{\mathcal T}$ using degree 10 Legendre polynomials. This guarantees that the $\mathcal H$-matrix approximation does not dominate the total approximation error. We use $3$-point Gaussian quadrature (namely the quadrature order $n=1$) to compute element integrals when assembling $\underline K$ and $\underline M$.  We also point out that since \texttt{H2Opus} currently supports square matrices only, we expand the target space $\mathbb V_t$ to $\mathbb V_s$ for the quadrature evaluations by the $\mathcal H^2$-approximation mentioned in Section~\ref{ss:fmm-h}, but only use  the subset of values of the matrix-vector product that correspond to interior quadrature points.

\subsection{Tests for the integral fractional Laplacian}
We first consider a classical fractional diffusion problem involving the integral fractional Laplacian, namely $a(x,y)\equiv 1$,  the order function $s(x) \equiv s$ is a constant in $(0,1)$, and the exterior domain $\Oe=\mathbb R^2\backslash\Oi$. So the solution $u$ satisfies $u\in \widetilde H^s(\Oi)$ and
\begin{equation}\label{e:int-f}
  \int_{\mathbb R^2}\int_{\mathbb R^2} \frac{(\widetilde u(x)-\widetilde u(y))(\widetilde v(x) - \widetilde v(y))}{|x-y|^{2+2s}} \diff y\diff x = \int_\Oi f v\diff x,
  \quad\text{for all } v\in \widetilde H^s(\Oi),
\end{equation}
where $\widetilde .$ denotes the zero extension from $\Oi$ to $\mathbb R^2$ and the fractional Sobolev space
\[
  \widetilde H^s(\Oi) := \{v\in L^2(\Oi) : \widetilde v\in H^s(\mathbb R^2)\} .
\]

\subsubsection{An extra step}
One bottleneck in generating a linear finite approximation for the above problem is to deal with the integral on the unbounded domain $\mathbb R^2\backslash\Oi$. Here we borrow the assembling strategy from \cite{acosta2017short} and briefly introduce the implementation below. We set an auxiliary triangulation $\mathcal T_B$ for a ball $\Omega_B$ centered at the origin with radius $R$ and containing the triangulation of $\Oi$. We set $R$ large enough so that the distance between $\Oi$ and $\partial \Omega_B$ is strictly positive. This guarantees that the patch $\mathcal S_\tau$ for each cell $\tau$ in $\Oi$ is contained in $B$. Whence, we follow \Cref{s:direct,s:hmatrix} exactly to assemble $\underline B$ and $\underline K$. To compute $\underline M$, according to \eqref{e:m}, we can split the discrete bilinear form $\mathcal M(.,.)$ as
\[
  \begin{aligned}
    \mathcal M(U,V) & = 2\sum_{\tau\in \mathcal T^{\text{int}}}\int_\tau U(x)V(x)\bigg(\int_{\Omega_B\backslash\mathcal S _\tau} \gamma(x,y)\diff y\bigg) \diff x                                 \\
                    & \qquad + 2\sum_{\tau\in \mathcal T^{\text{int}}}\int_\tau U(x)V(x)\bigg(\underbrace{\int_{\mathbb R^2\backslash \Omega_B} \gamma(x,y)\diff y}_{:=\rho_B(x)}\bigg) \diff x .
  \end{aligned}
\]
Denote $\underline M^\text{in}$ and $\underline M^{\text{out}}$ the associated \emph{weighted mass} matrices for the two bilinear forms on the right-hand side of the equation above. We apply the fast multipole approximation technique of \Cref{s:fmm} to $\underline M^{\text{in}}$. For $\underline M^{\text{out}}$, we use the fact that $\rho_B$ is radial in $\mathbb R^2\backslash\Omega_B$ and thus rewrite $\rho_B$ in polar coordinates with (cf. \cite[Section~A.5]{acosta2017short})
\begin{equation}\label{e:rho}
  \rho_B(x) = \frac1{2s} \int_0^{2\pi} \frac1{t(\theta ,x)^{2s}} \diff \theta,
\end{equation}
where
\[
  t(\theta ,x) := \sqrt{\mu^2 + R^2 - |x|^2} - \mu, \quad\text{with } \mu=x_1\cos\theta + x_2\sin\theta .
\]
Thus when assembling $\underline M^{\text{out}}$ by using quadrature formulas for each $\tau\in \Oi$, we evaluate $\rho_B$ at each quadrature point by approximating the integral in \eqref{e:rho} by numerical integration. Here we use a $9$-point Gaussian quadrature formula. The stiffness matrix corresponding to the weak problem \eqref{e:int-f} is now decomposed into four sub-matrices, \ie
\[
  \underline A = \underline B + \underline K + \underline M^{\text{in}} + \underline M^{\text{out}}.
\]

\subsubsection{Simulation and results}
We set $\Oi$ to be the unit ball. For the auxiliary ball $\Omega_B$, we set its radius $R=1.1$. We shall test the convergence of the finite element approximation by using the well-known analytic solution
\[
  u(x) = \frac{2^{-2s}}{\Gamma(1+s)\Gamma(1+s)} (1-|x|^2)^s,\quad\text{in }\Oi
\]
so that $f=1$ in $\Oi$. Starting from a coarse grid $\mathcal T_1$ for $\Omega_B$, we generate a sequence for meshes $\{\mathcal T_j\}_{j=1}^6$ by refining the mesh globally. We set $\overline s = 0.9$ for the computation of the sub-matrix $\underline B$.

The left panel of \Cref{f:disk} reports the $L^2(\Oi)$-error between $u$ and its finite element approximation $U_j$ against the number of degrees of freedom when $s=0.7$. The slope of the log-log error plot implies that $U_j$ converges to $u$ in the first order, which is the optimal rate that can be reached and is limited only by the reduced regularity of the solution itself, which has singular derivatives at the boundary.
The right panel of \Cref{f:disk} reports the CPU time for assembling the sub-matrices $\underline B$, $\underline K$ and $\underline M^{\text{in}}$, respectively. As the number of degrees of freedom increases, we observe a linear complexity for all three assembly routines.

\begin{figure}[ht]
  \centering
  \includegraphics[width=\textwidth]{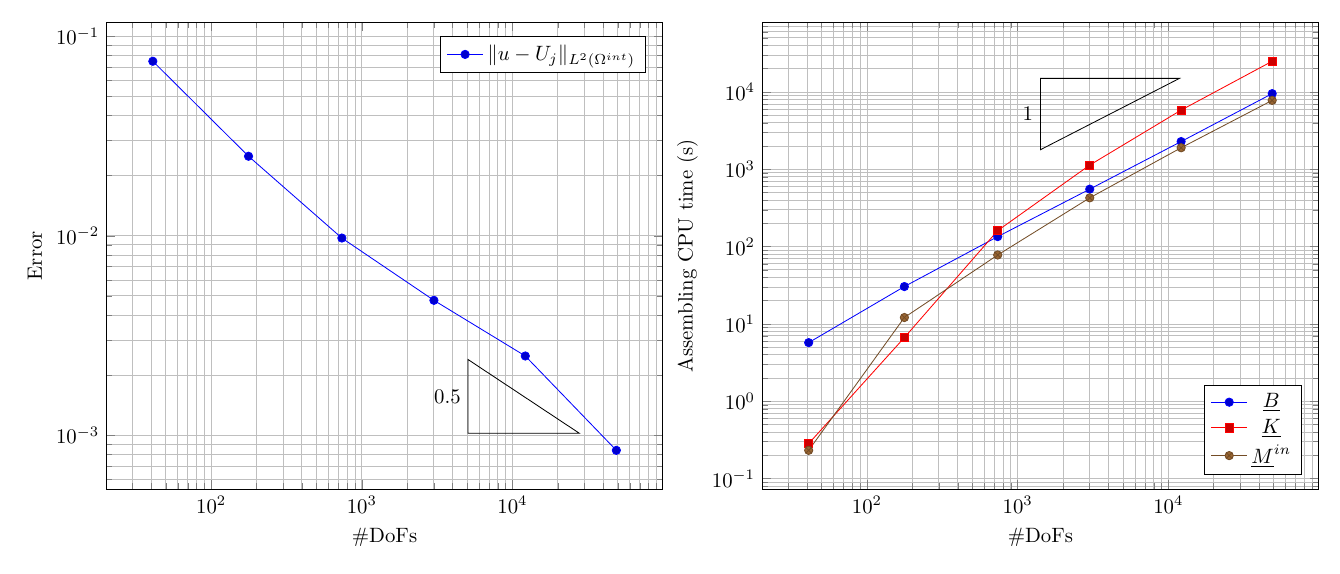}
  \caption{Integral fractional Laplacian test: (left) $L^2$-error between $u$ and $U_j$ against the degrees of freedoms ($\#\text{DoFs}$) and (right) CPU times for each assembly routines versus degrees of freedoms $\#\text{DoFs}$. Linear complexity in time for all three assembly routines is confirmed.}
  \label{f:disk}
\end{figure}

\Cref{f:h-structure} depicts the tree structure of the $\mathcal H$-matrix for $\underline K$ (left panel). Here the blocks in red are computed directly while the blocks in green are approximated by the low-rank matrices based on the Lagrange interpolation. The right plot of \Cref{f:h-structure} illustrates the tree structure of the $\mathcal H$-matrix that is used to compute the density function $\rho_{\mathcal T}$; see Section~\ref{s:fmm} for details. Note that the matrix on the right is larger than the one on the left because it includes all interior and exterior degrees of freedom and we are using a $3$-point Gaussian quadrature for the density function $\rho_{\mathcal T}$ in \eqref{e:m}.

\begin{figure}[H]
  \centering
  \begin{tabular}{cc}
    \hspace*{-10pt}
    \includegraphics[width=0.5\textwidth]{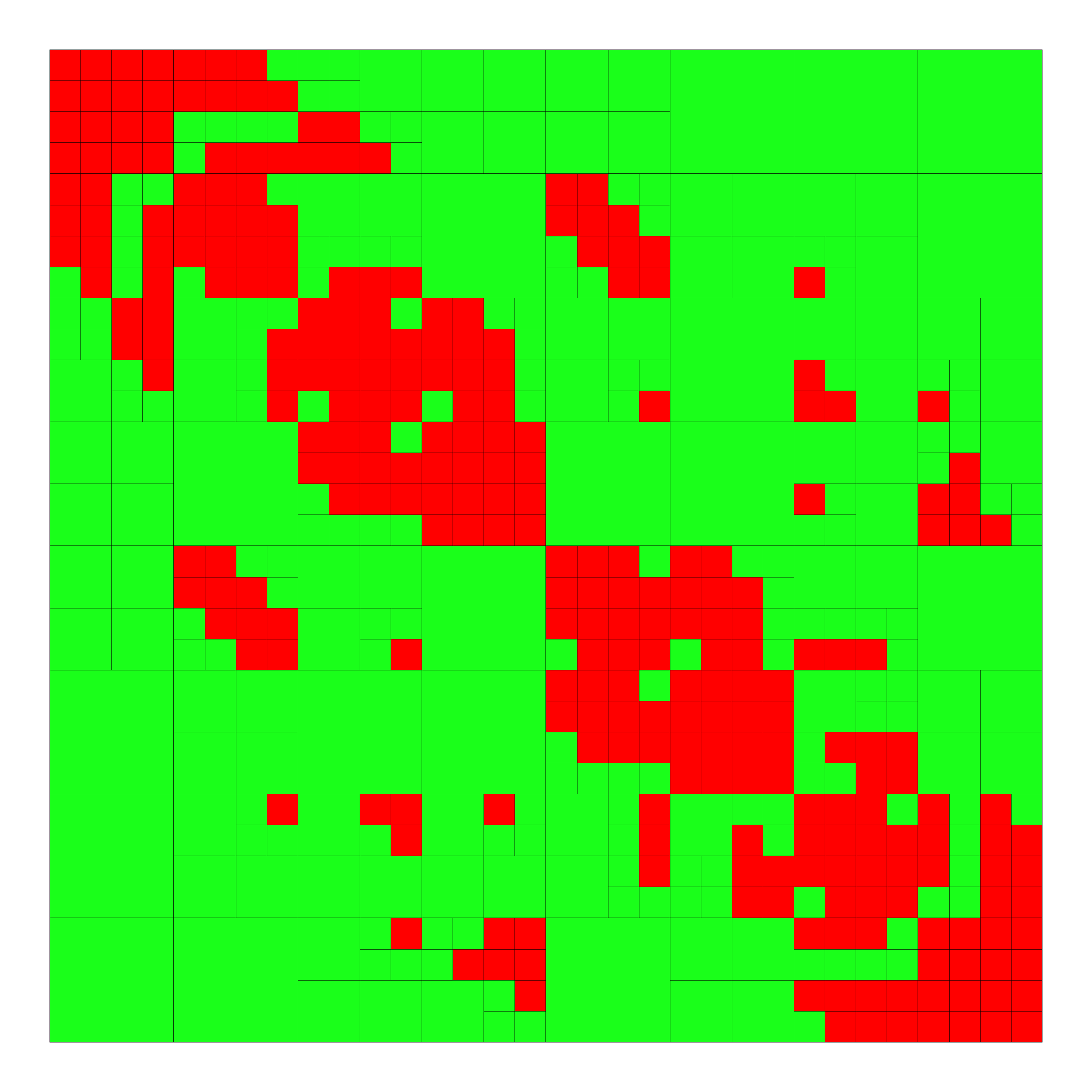} &
    \hspace*{-10pt}
    \includegraphics[width=0.5\textwidth]{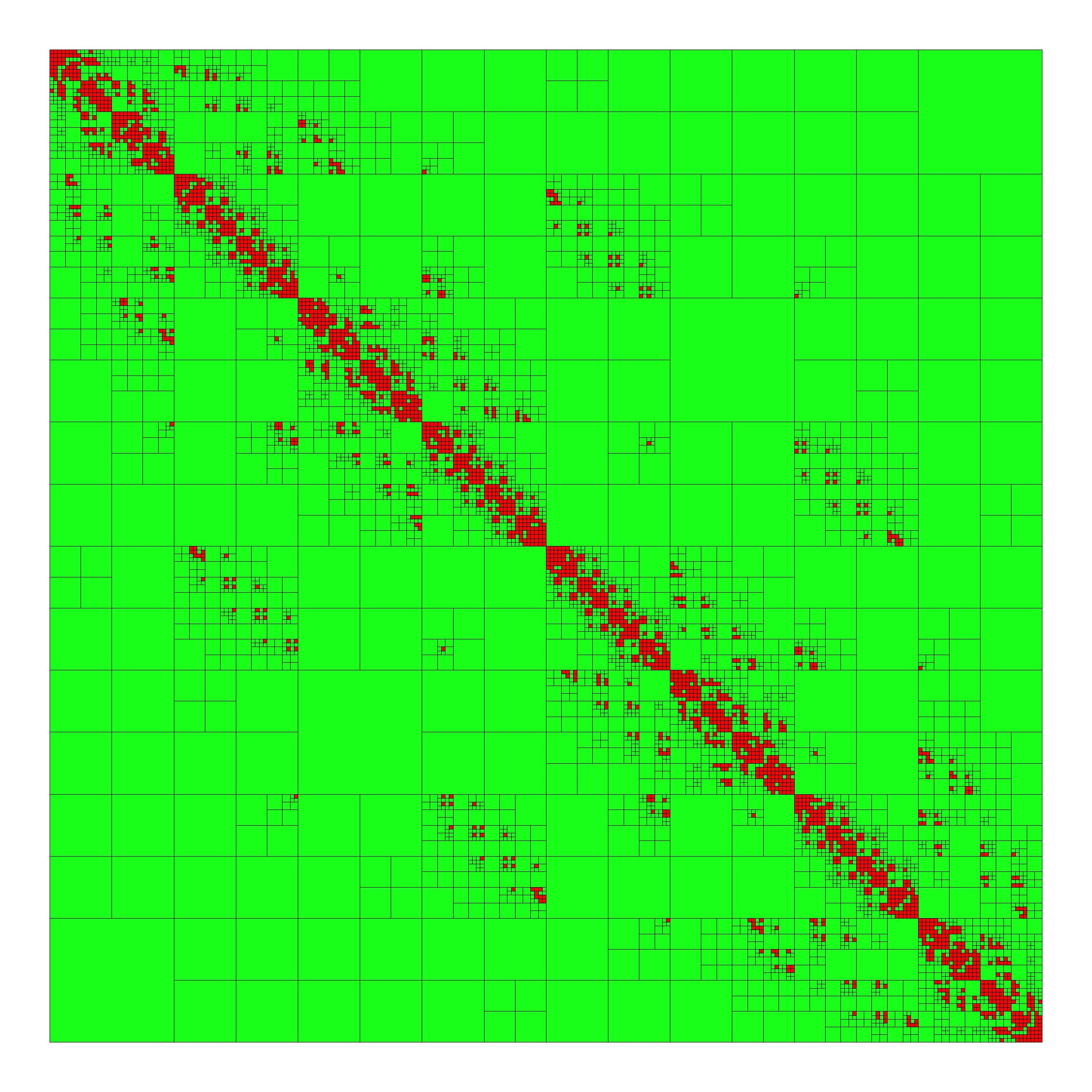}
  \end{tabular}
  \caption{Integral fractional Laplacian test: (left) $\mathcal{H}^2$-matrix structure of $\underline K$ for the mesh $\mathcal T_4$ and (right) the corresponding $\mathcal{H}^2$-matrix structure for the quadrature evaluations for $\rho_{\mathcal T}$.  Matrix blocks in red are computed directly while the blocks in green are approximated by low-rank factorizations based on polynomial Lagrange interpolation of the kernel.}
  \label{f:h-structure}
\end{figure}

\subsection{Tests for variable order}
Next we test our algorithm with variable-order FDEs. We let constant $s^*\in (\tfrac12 ,1)$ be the so-called \emph{background order} in $\Omega$. This means that the variable order function $s(x) = s^*$ in $\Oe$. In $\Oi$, we consider a tensor product \emph{bump} function that is supported on a square in $\Oi$ and centered at the point $(x^c_1,x^c_2)$ with the size $\ell$. Specifically, we let
\begin{equation}\label{e:variable-order}
  s(x) = s^* + \eta \times \textrm{bump}(x_1-x_1^c, \ell)\times\textrm{bump}(x_2-x_2^c;\ell)        ,
\end{equation}
where
\[
  \textrm{bump}(x;\ell) =
  \left\{
  \begin{aligned}
     & \exp(-\frac{1}{1-r^2}),\, r=\frac{2x}{\ell}, &  & |r|<1,   \\
     & 0,                                           &  & |r|\ge 1 \\
  \end{aligned}
  \right.
\]
and $\eta$ is a fixed constant satisfying that $s(x)\in [\tfrac12,1)$. For simplicity, we fix the diffusion coefficient $a(x,y)\equiv 1$.

\begin{figure}[ht]
  \begin{center}
    \begin{tabular}{ccc}
      \includegraphics[width=0.3\textwidth]{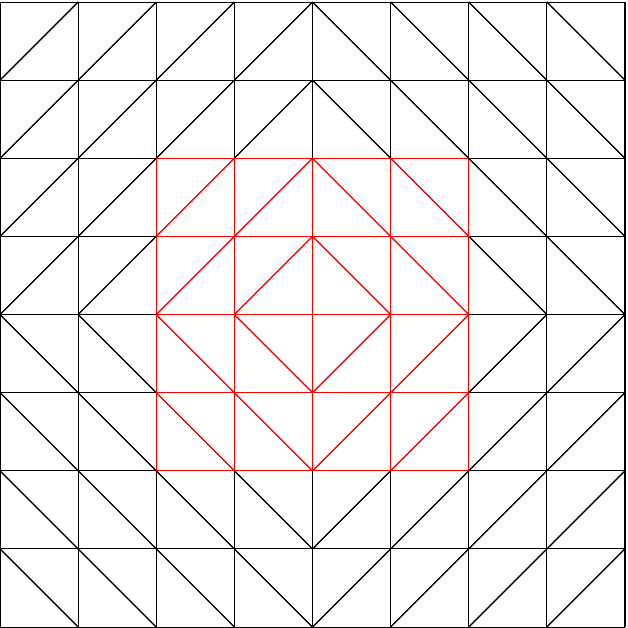} &
      \phantom{AAAA}                                         &
      \includegraphics[width=0.4\textwidth]{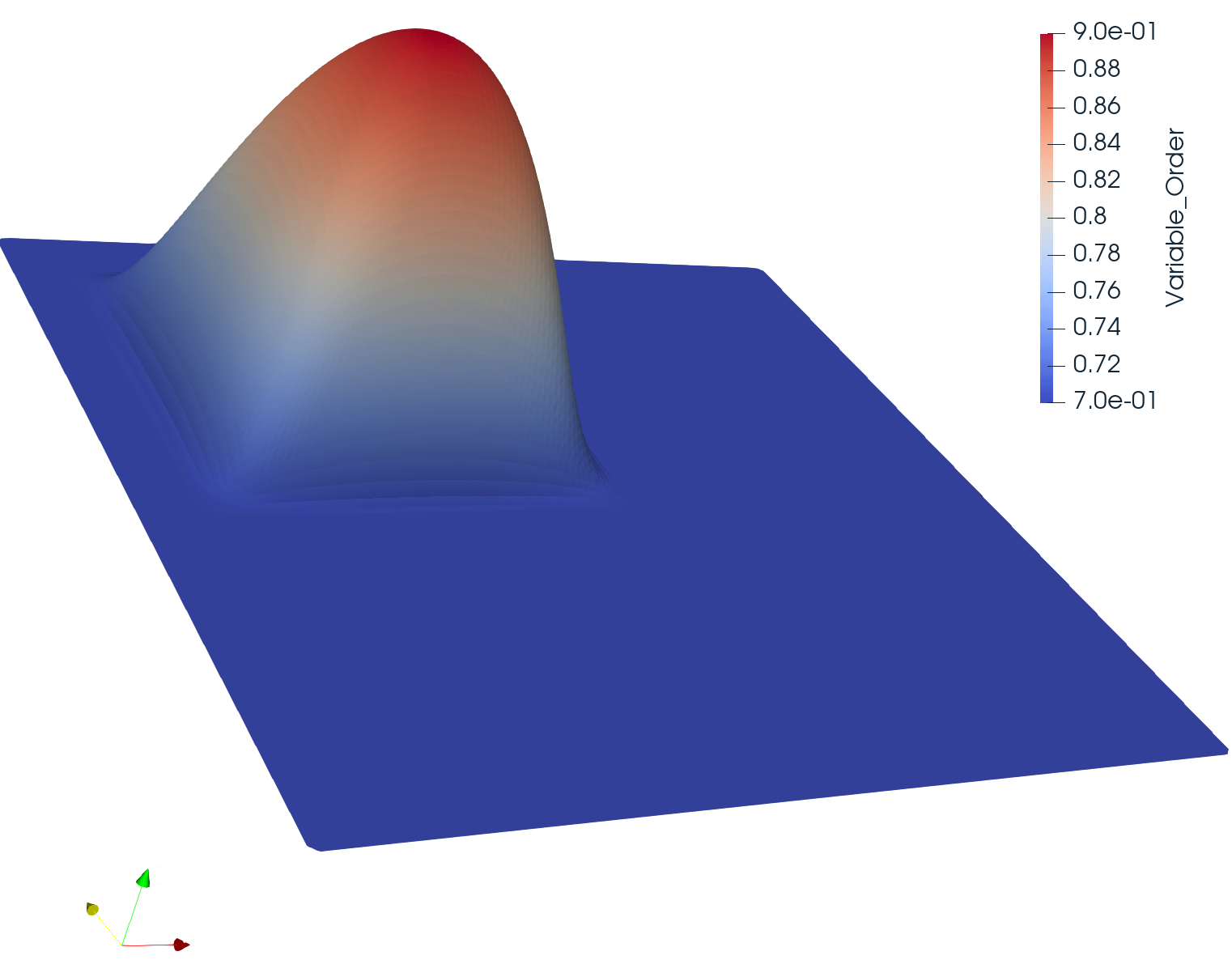}
    \end{tabular}
  \end{center}
  \caption{(Left) the coarsest uniform grid for the domain $\Omega = (-2,2)^2$. The grid for the interior domain $\Oi = (-1,1)^2$ is marked in red. (Right) variable order function in $\Oi$ with $s^* =0.7$, $\eta=0.2$, $\ell = 1.0$ and $x^c=(-0.4,0.4)$.}
  \label{f:grid}
\end{figure}

Our computational domain $\Omega$ is set to be a square $(-2,2)^2$ and the interior domain $\Oi$ is set to be $(-1,1)^2$. We construct a sequence of uniform grids $\{\mathcal T_j\}_{j=1}^6$ with the mesh size $h_j = \sqrt{2}/j^{j+1}$ generated by globally refining the coarse grid $\mathcal T_1$; see the coarsest grid $\mathcal T_1$ in Figure~\ref{f:grid}.

\begin{remark}\label{r:quasi-uniform-test}
  We note that the subdivision in $\Oe$ mainly contributes to the assembly of matrix $\underline M$ by computing the density function $\rho_{\mathcal T}(x)$ as in \eqref{e:m}. By utilizing the decay property of $\gamma(x,y)$, graded meshes in $\Oe$ can be used in order to reduce the computational cost. This strategy will be explored in future work. Here we rely on quasi-uniform meshes in $\Oe$ to guarantee that the error from $\rho_{\mathcal T}$ does not affect the total error.
\end{remark}

For spatially varying fractional order, we do not have analytic solutions to examine the rate of convergence of our numerical scheme. As an alternative, we perform a comparison test by computing the difference between the finite element solution $U_j$ on $\mathcal T_j$ and a finite difference approximation $\widetilde U_j$ developed in \cite{alzahrani2022space} using a cartesian grid $\widetilde{\mathcal T}_j$ with the same node locations. We also perform self-convergence tests.

\begin{figure}[ht]
  \centering
  \includegraphics[width=\textwidth]{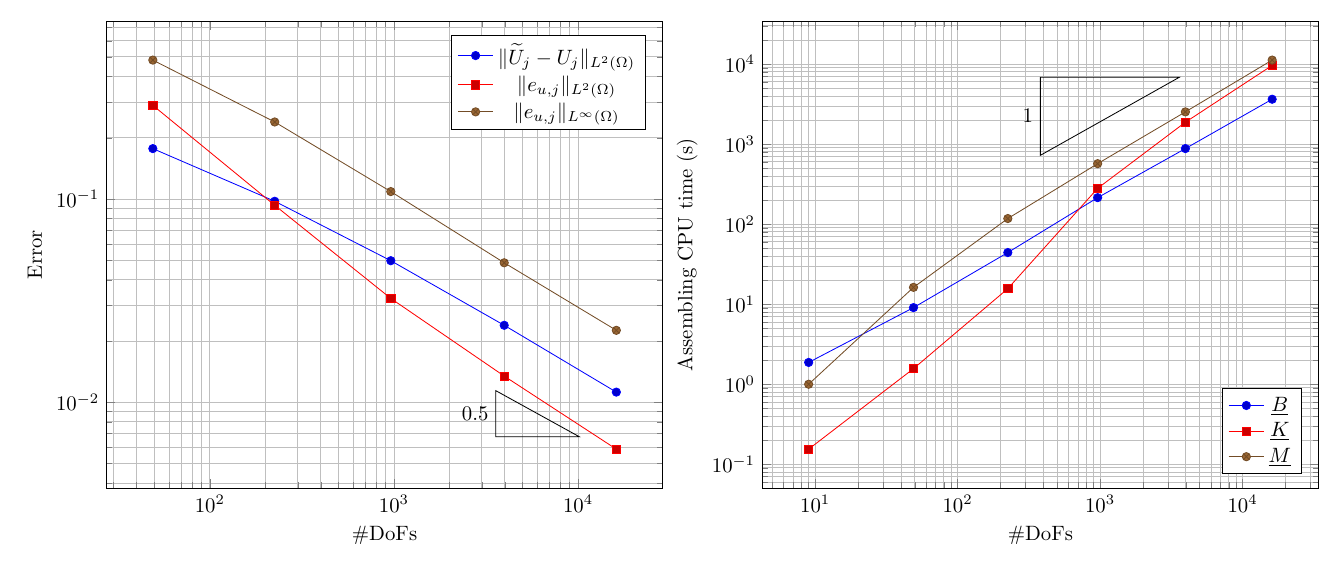}
  \caption{Variable-order comparison test: (left) $L^2$-error between $u$ and $U_j$ against degrees of freedoms $\#\text{DoFs}$ and the self-convergence of $U_j$ in $L^2$ and $L^\infty$ norms; (right) CPU times for assembling routines versus degrees of freedoms $\#\text{DoFs}$. Linear complexity in time is confirmed.}
  \label{f:comparsion}
\end{figure}

\enlargethispage*{15pt}

The left plot of Figure~\ref{f:comparsion} reports the $L^2(\Oi)$-error decay between $U_j$ and $\widetilde U_j$ for $f\equiv 20$ and for the variable order function $s(x)$ in \eqref{e:variable-order} with $s^* = 0.7$, $\eta = 0.2$, $\ell=1.0$, and $(x_c,y_c) = (-0.4,0.4)$; see Figure~\ref{f:grid}. We also report the self-convergence of $U_j$, namely the error $e_{u,j} = U_{j+1}-U_j$, in both $L^2(\Oi)$ and $L^\infty(\Oi)$ norms, by estimating the error on each grid by using the next finer grid as the reference solution.
It is seen that all three errors exhibit a first order decay.
Such error behavior is similar to the constant-order integral fractional Laplacian case when $s \ge 1/2$, and 
might be expected here since $\underline{s} \ge 1/2$. 
We also observe the singular behavior of the solution at $\partial\Oi$; see the approximate solution along $x_1=-0.4$ in the right plot of Figure~\ref{f:comparsion-sol}. 
The right panel of Figure~\ref{f:comparsion} shows the performance of the assembly routines for matrices $\underline B$, $\underline K$ and $\underline M$. We again observe a linear complexity in time for each assembly procedure, though the time for $\underline M$ is relatively large due the quasi-uniform triangulation of $\Omega$ as mentioned in Remark~\ref{r:quasi-uniform-test}.

\begin{figure}[H]
  \begin{center}
    \begin{tabular}{lr}
      \includegraphics[width=0.4\textwidth]{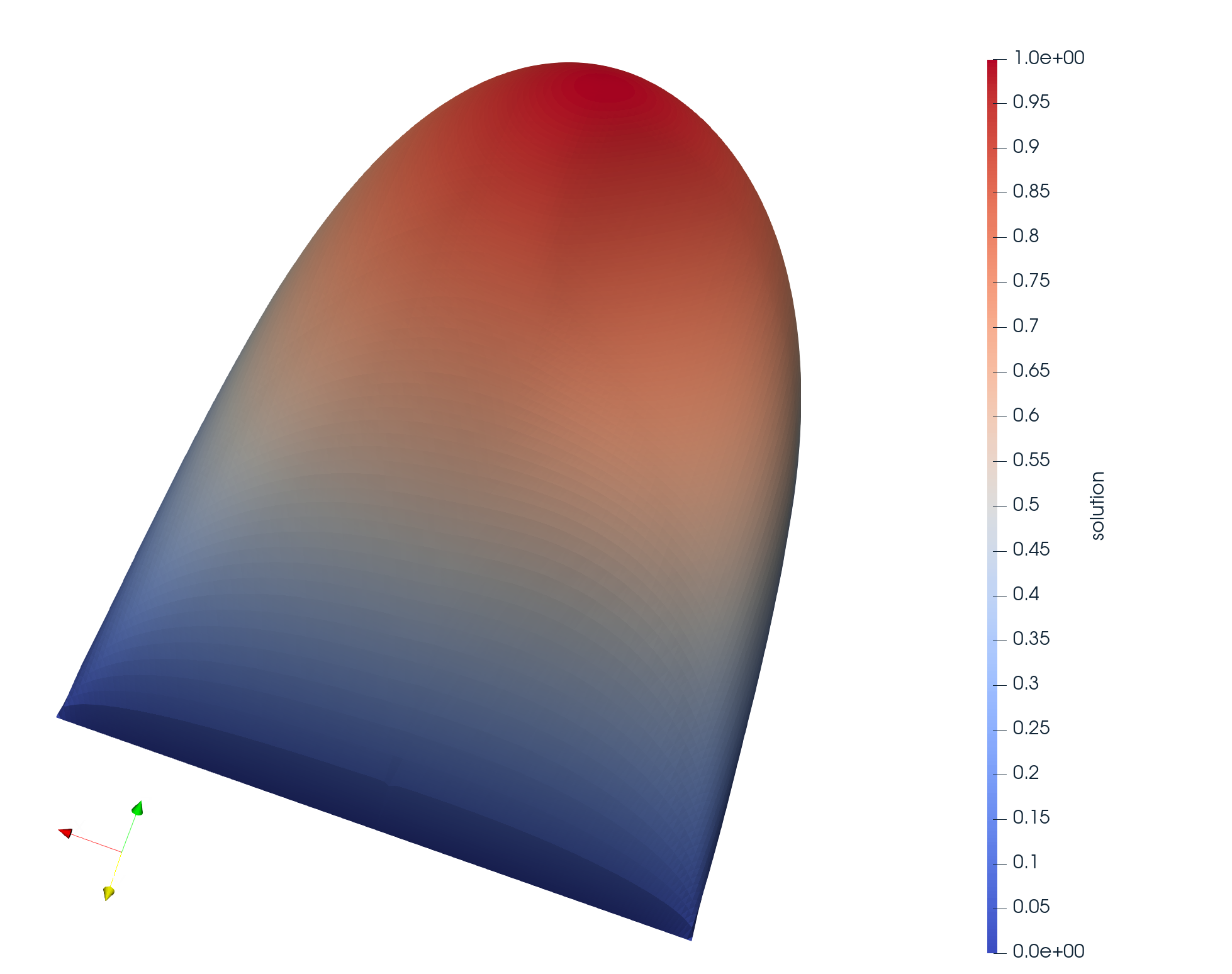} &
      \includegraphics[width=0.35\textwidth]{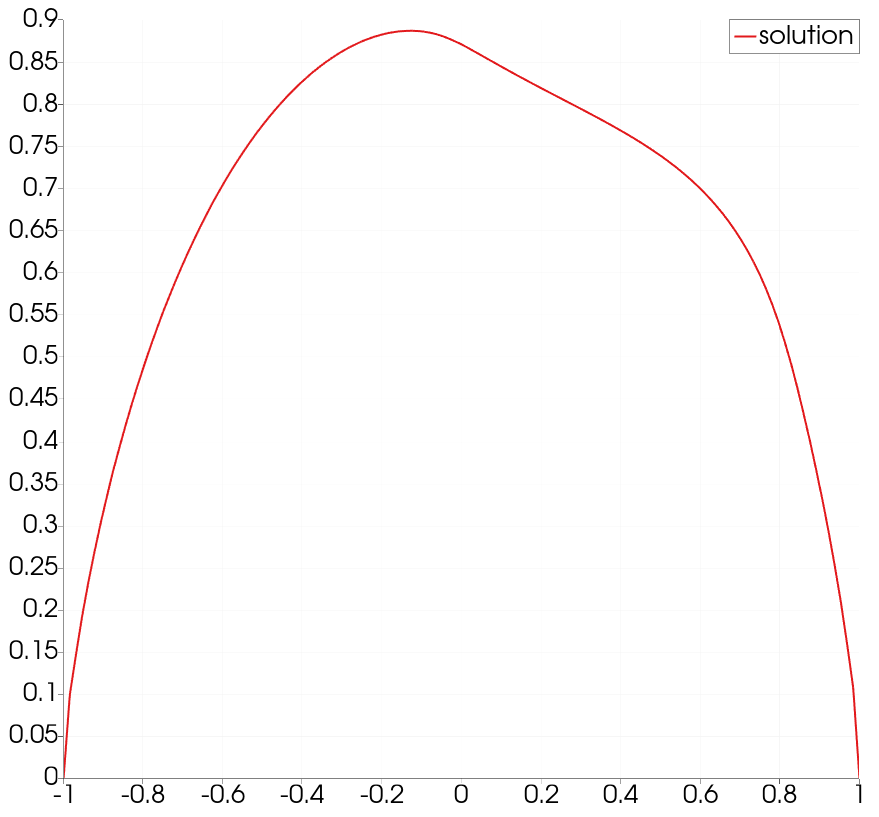}
    \end{tabular}
  \end{center}
  \caption{Variable-order comparison test: (left) the finite element approximation $U_6$ ($\#\text{DoFs} = 16129$) and (right) the approximation along $x_1=-0.4$.}
  \label{f:comparsion-sol}
\end{figure}

\Cref{f:comparsion-sol} shows the finite element approximation $U_6$ with $16129$ degrees of freedom in $\Oi$ ($131072$ cells total in $\Omega$). In the left panel we see a faster diffusion rate in the variable-order region $[-0.9,0.1]\times[-0.1,0.9]$. This diffusive behavior of the solution can also be observed along the diagonal $x_2=-0.4$ shown in the right panel of \Cref{f:comparsion-sol}. 

In \Cref{f:comparsion-negative}, we test the same problem using the same parameters except that the coefficient bump function is negative, namely $\eta=-0.2$. As shown in the left panel, we again obtain first-order convergence in $L^2(\Oi)$ when comparing against the solution obtained from the finite difference method. In the right panel, we instead observe a less diffusive behavior of the solution in the bump region. 
Figure~\ref{fig:condition} displays the estimated condition number of the system (computed from the CG iterates) against the number of degrees of freedom. We observe that the condition number is nearly $O(h^{-2s^*})$, where $s^*$ denotes the background order.

\begin{figure}[ht]
  \begin{tabular}{p{0.5\textwidth} p{0.5\textwidth}}
    \vspace{0pt} \includegraphics[width=0.49\textwidth]{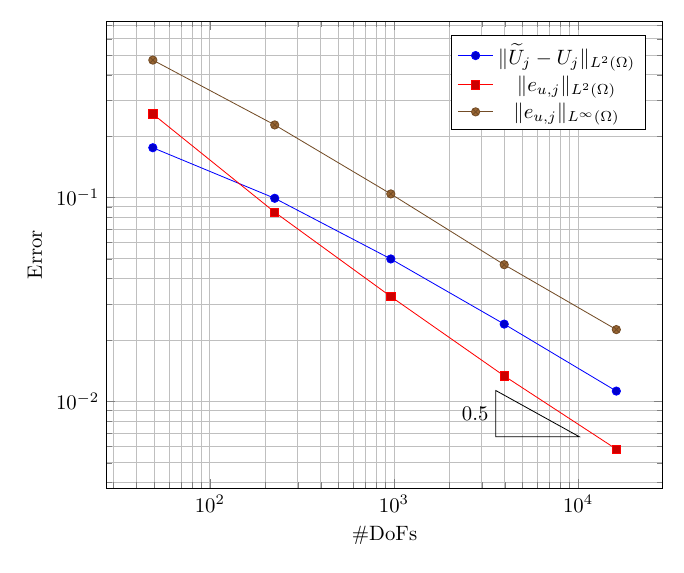} &
    \vspace{0pt} \includegraphics[width=0.4\textwidth]{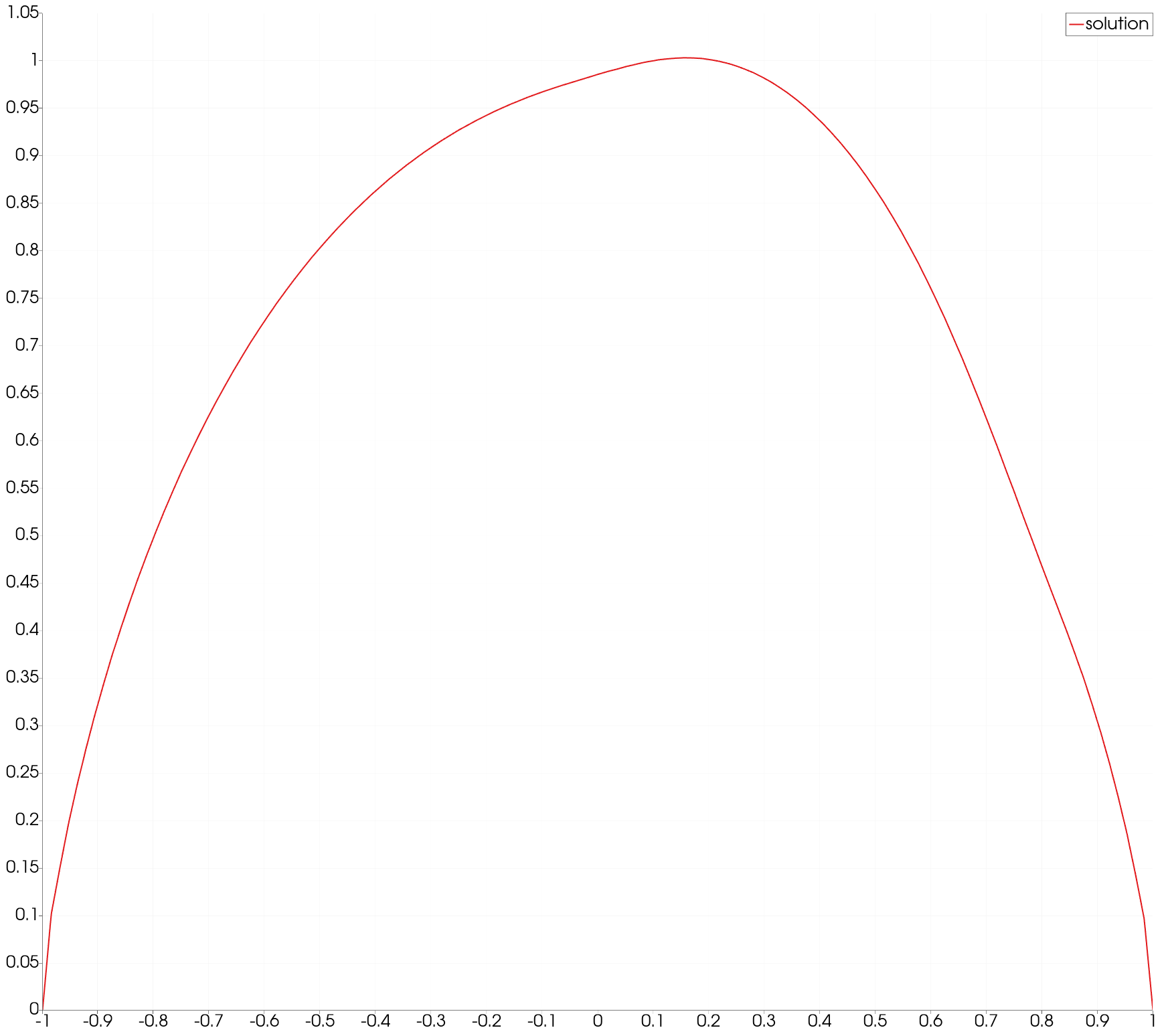}                   
  \end{tabular}
  \caption{Variable-order comparison test: (left) $L^2$-error between $u$ and $U_j$ against degrees of freedoms $\#\text{DoFs}$ and the self-convergence of $U_j$ in $L^2$ and $L^\infty$ norms; (right) the finite element approximation $U_6$ ($\#\text{DoFs} = 16129$) along $x_1=-0.4$.}
  \label{f:comparsion-negative}
\end{figure}

\begin{figure}[H]
\begin{center}
\includegraphics[width=0.49\textwidth]{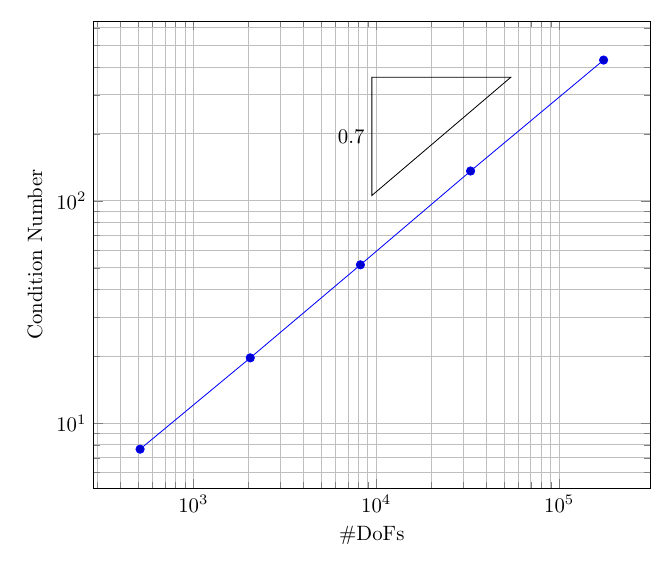}
\end{center}	
\caption{Estimated condition number of the system matrix.}
\label{fig:condition}
\end{figure}

The first-order convergence rate $O(h)$ observed in the examples above is primarily due to the lack of regularity of the solution itself and is essentially the best that can be obtained for these solutions because of the singular derivatives at the boundary. One may wonder whether the proposed linear finite element discretization can produce second-order convergence when the solution has sufficient regularity.

To verify higher-order convergence, and since we cannot readily manufacture a solution analytically for variable order problems, we consider the following numerical alternative. We start from a given smooth solution and apply the forward operator discretized by a finite difference scheme to obtain a right-hand side. We then use this right-hand side in the finite element solver to recover the given solution.

\begin{figure}[ht]
  \centering
  \includegraphics[width=\textwidth]{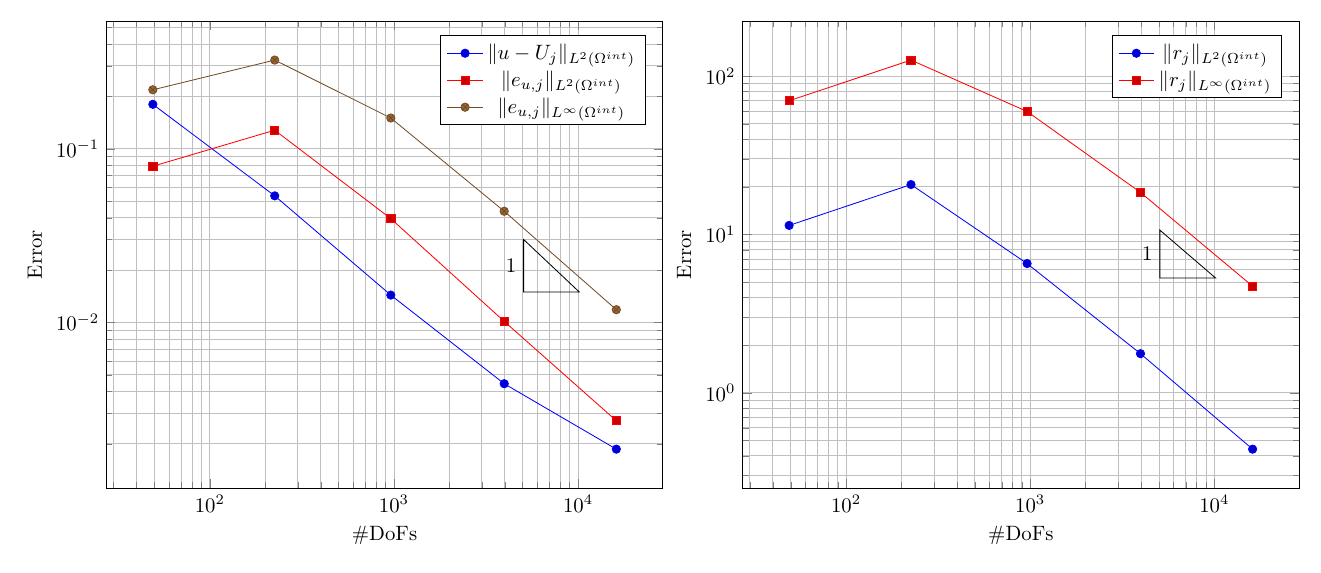}
  \caption{Variable-order test with smooth solutions: (left) $L^2$-error between $u$ and $U_j$ against the number of degrees of freedom $\#\text{DoFs}$ and the self-convergence of $U_j$ in $L^2$ and $L^\infty$ norms. (right) self-convergence of the right-hand side approximation using a finite difference method.}
  \label{f:comparsion-rhs}
\end{figure}

Specifically, we seek the solution $u$ with right-hand side data $F_j = I_j A_{\text{FD},j}u$, where $A_{\text{FD},j}u$ is the discrete right hand side data produced by the finite difference method \cite{alzahrani2022space} using the cartesian grid $\widetilde{\mathcal T}_j$ and $I_j$ is the Lagrange nodal interplant on $\mathbb V(\mathcal T_j)$. We fix $u$ to be the smooth function $u(x,y) = (1-x^2)^{15}(1-y^2)^{15}$ and set the same variable order function $s(x)$ in the previous variable order test.
\Cref{f:comparsion-rhs} shows the results of this test. The right panel of \Cref{f:comparsion-rhs} shows the convergence of the right hand side data as the mesh is refined. While in general only a first-order rate of convergence for $r_j:=F_{j+1}-F_j$ in $L^2(\Oi)$-norm is expected due to the singularity at the boundary, we obtain here second-order convergence since the solution is sufficiently smooth.
The left plot of Figure~\ref{f:comparsion-rhs} shows that the proposed finite element approximation can indeed obtain second-order convergence when solution regularity allows it.


\section{Conclusions}
\label{sec:conclusions}

We presented an asymptotically optimal finite element method for modeling non-local fractional diffusion operators with spatial variation in fractional oder and material coefficients and in general geometries. In the finite-element formulation, triangle pairs in the spatial mesh are the basic units that generate elemental stiffness matrices to be assembled into a global stiffness matrix.  
We address the singularities in evaluating touching triangle pairs though specialized mapping and quadrature schemes designed to handle the variable-order case. 
The computational complexity due to the quadratic number of interacting triangle pairs is overcome through (i) the construction of the hierarchical matrix approximations for representing the effect of every interior node on all other ones, and (ii) a generalized variable-order fast multipole method that computes the cumulative effect of all triangles on every interior node. 
The overall complexity for building the complete discrete operator is optimal, $O(N)$, both in memory and in operations. We show the consistency of the method and the ability to control its accuracy through the number of quadrature points in the direct integrations and the degree of the polynomial interpolant of the kernel in the far field.
Numerical experiments verify the accuracy and complexity of the methods proposed. The techniques are general and 
should apply to a broader class of nonlocal kernels.

There are a number of directions we are pursing in further work.
First, the current work focused on the construction and application of the discrete operator. We are developing scalable preconditioners to allow the iterative solution of general fractional diffusion problems and are pursuing a geometric multilevel strategy that would allow end-to-end solutions in linear complexity. 
Second, the computations involved in building and applying the operator have high arithmetic intensity and substantial concurrency. We are developing performant implementations that are GPU-accelerated and can take advantage of multiple cores in a node and multiple distributed-memory nodes. This should produce dramatic improvements in absolute efficiency and make it quite practical to work with fractional operators in various application domains. 
We are also interested in extending the present finite element development to the case of anisotropic variability in fractional order and coefficients as well to three-dimensional geometry, a setting that has not yet been addressed sufficiently in the literature and where graded meshes will be particularly important.
Finally, we are working on incorporating the forward solver in the inner loop of an inverse problem that seeks to recover the spatial distribution of fractional order and coefficients. In this context, the flexibility and linear complexity in memory and operations of the forward solver are essential from a practical point of view. 
We plan to report on these developments elsewhere.



\bibliographystyle{siam}
\bibliography{main}

\begin{thebibliography}{10}

\bibitem{acosta2017short}
{\sc G.~Acosta, F.~M. Bersetche, and J.~P. Borthagaray}, {\em A short {FE}
  implementation for a 2d homogeneous {D}irichlet problem of a fractional
  {L}aplacian}, Comput. Math. Appl., 74 (2017), pp.~784--816.

\bibitem{acosta2017fractional}
{\sc G.~Acosta and J.~P. Borthagaray}, {\em A fractional {L}aplace equation:
  regularity of solutions and finite element approximations}, SIAM J. Numer.
  Anal., 55 (2017), pp.~472--495.

\bibitem{ainsworth2017aspects}
{\sc M.~Ainsworth and C.~Glusa}, {\em Aspects of an adaptive finite element
  method for the fractional {L}aplacian: a priori and a posteriori error
  estimates, efficient implementation and multigrid solver}, Comput. Methods
  Appl. Mech. Engrg., 327 (2017), pp.~4--35.

\bibitem{ainsworth2018towards}
\leavevmode\vrule height 2pt depth -1.6pt width 23pt, {\em Towards an efficient
  finite element method for the integral fractional {L}aplacian on polygonal
  domains}, in Contemporary computational mathematics---a celebration of the
  80th birthday of {I}an {S}loan. {V}ol. 1, 2, Springer, Cham, 2018,
  pp.~17--57.

\bibitem{alzahrani2022space}
{\sc H.~Alzahrani, G.~Turkiyyah, O.~Knio, and D.~Keyes}, {\em Space-fractional
  diffusion with variable order and diffusivity: discretization and direct
  solution strategies}, Commun. Appl. Math. Comput., 4 (2022), pp.~1416--1440.

\bibitem{alzahrani21}
{\sc H.~H. Alzahrani, M.~Lucchesi, K.~Mustapha, O.~P.~L. Ma{\^\i}tre, and O.~M.
  Knio}, {\em Bayesian calibration of order and diffusivity parameters in a
  fractional diffusion equation}, Journal of Physics Communications, 5 (2021),
  p.~085014.

\bibitem{deal.II}
{\sc D.~Arndt, W.~Bangerth, M.~Feder, M.~Fehling, R.~Gassm{\"o}ller,
  T.~Heister, L.~Heltai, M.~Kronbichler, M.~Maier, P.~Munch, et~al.}, {\em The
  deal. ii library, version 9.4}, Journal of Numerical Mathematics, 30 (2022),
  pp.~231--246.

\bibitem{bauer2022kernel}
{\sc M.~Bauer, M.~Bebendorf, and B.~Feist}, {\em Kernel-independent adaptive
  construction of {$\mathcal {H}^{2}$}-matrix approximations}, Numer. Math.,
  150 (2022), pp.~1--32.

\bibitem{bonito2019numerical}
{\sc A.~Bonito, W.~Lei, and J.~E. Pasciak}, {\em Numerical approximation of the
  integral fractional {L}aplacian}, Numer. Math., 142 (2019), pp.~235--278.

\bibitem{borm2010efficient}
{\sc S.~B{\"o}rm}, {\em Efficient numerical methods for non-local operators:
  $\mathcal{H}^{2}$-matrix compression, algorithms and analysis}, vol.~14,
  European Mathematical Society, 2010.

\bibitem{borthagaray2022besov}
{\sc J.~P. Borthagaray and R.~H. Nochetto}, {\em Besov regularity for the
  {D}irichlet integral fractional {L}aplacian in {L}ipschitz domains}, J.
  Funct. Anal., 284 (2023), p.~Paper No. 109829.

\bibitem{boukaram2020hierarchical}
{\sc W.~Boukaram, M.~Lucchesi, G.~Turkiyyah, O.~Le~Ma\^{i}tre, O.~Knio, and
  D.~Keyes}, {\em Hierarchical matrix approximations for space-fractional
  diffusion equations}, Comput. Methods Appl. Mech. Engrg., 369 (2020),
  pp.~113191, 22.

\bibitem{boukaram19a}
{\sc W.~Boukaram, G.~Turkiyyah, and D.~Keyes}, {\em Hierarchical matrix
  operations on {GPU}s: Matrix-vector multiplication and compression}, ACM
  Transactions on Mathematical Software, 45 (2019), pp.~3:1--3:28.

\bibitem{caffarelli2007extension}
{\sc L.~Caffarelli and L.~Silvestre}, {\em An extension problem related to the
  fractional {L}aplacian}, Comm. Partial Differential Equations, 32 (2007),
  pp.~1245--1260.

\bibitem{chen17}
{\sc P.~Chen, U.~Villa, and O.~Ghattas}, {\em Hessian-based adaptive sparse
  quadrature for infinite-dimensional bayesian inverse problems}, Computer
  Methods in Applied Mechanics and Engineering, 327 (2017), pp.~147--172.
\newblock Advances in Computational Mechanics and Scientific Computation---the
  Cutting Edge.

\bibitem{tankov2003financial}
{\sc R.~Cont and P.~Tankov}, {\em Financial modelling with jump processes},
  Chapman \& Hall/CRC Financial Mathematics Series, Chapman \& Hall/CRC, Boca
  Raton, FL, 2004.

\bibitem{contreras2016multi}
{\sc A.~A. Contreras, O.~P. {Le Ma{\^\i}tre}, W.~Aquino, and O.~M. Knio}, {\em
  Multi-model polynomial chaos surrogate dictionary for bayesian inference in
  elasticity problems}, Probabilistic Engineering Mechanics, 46 (2016),
  pp.~107--119.

\bibitem{davis1975interpolation}
{\sc P.~J. Davis}, {\em Interpolation and approximation}, Dover Publications,
  Inc., New York, 1975.
\newblock Republication, with minor corrections, of the 1963 original, with a
  new preface and bibliography.

\bibitem{d2022fractional}
{\sc M.~D'Elia and C.~Glusa}, {\em A fractional model for anomalous diffusion
  with increased variability: analysis, algorithms and applications to
  interface problems}, Numer. Methods Partial Differential Equations, 38
  (2022), pp.~2084--2103.

\bibitem{d2013fractional}
{\sc M.~D'Elia and M.~Gunzburger}, {\em The fractional {L}aplacian operator on
  bounded domains as a special case of the nonlocal diffusion operator},
  Comput. Math. Appl., 66 (2013), pp.~1245--1260.

\bibitem{d2021cookbook}
{\sc M.~D'Elia, M.~Gunzburger, and C.~Vollmann}, {\em A cookbook for
  approximating {E}uclidean balls and for quadrature rules in finite element
  methods for nonlocal problems}, Math. Models Methods Appl. Sci., 31 (2021),
  pp.~1505--1567.

\bibitem{du2012analysis}
{\sc Q.~Du, M.~Gunzburger, R.~B. Lehoucq, and K.~Zhou}, {\em Analysis and
  approximation of nonlocal diffusion problems with volume constraints}, SIAM
  Rev., 54 (2012), pp.~667--696.

\bibitem{du2013nonlocal}
\leavevmode\vrule height 2pt depth -1.6pt width 23pt, {\em A nonlocal vector
  calculus, nonlocal volume-constrained problems, and nonlocal balance laws},
  Math. Models Methods Appl. Sci., 23 (2013), pp.~493--540.

\bibitem{duo2018novel}
{\sc S.~Duo, H.~W. van Wyk, and Y.~Zhang}, {\em A novel and accurate finite
  difference method for the fractional {L}aplacian and the fractional {P}oisson
  problem}, J. Comput. Phys., 355 (2018), pp.~233--252.

\bibitem{duo2019accurate}
{\sc S.~Duo and Y.~Zhang}, {\em Accurate numerical methods for two and three
  dimensional integral fractional {L}aplacian with applications}, Comput.
  Methods Appl. Mech. Engrg., 355 (2019), pp.~639--662.

\bibitem{faustmann2022weighted}
{\sc M.~Faustmann, C.~Marcati, J.~M. Melenk, and C.~Schwab}, {\em Weighted
  {A}nalytic {R}egularity for the {I}ntegral {F}ractional {L}aplacian in
  {P}olygons}, SIAM J. Math. Anal., 54 (2022), pp.~6323--6357.

\bibitem{gatto2015numerical}
{\sc P.~Gatto and J.~S. Hesthaven}, {\em Numerical approximation of the
  fractional laplacian via {\$}{\$}hp{\$}{\$}-finite elements, with an
  application to image denoising}, Journal of Scientific Computing, 65 (2015),
  pp.~249--270.

\bibitem{glusa2022asymptotically}
{\sc C.~Glusa, M.~D'Elia, G.~Capodaglio, M.~Gunzburger, and P.~B. Bochev}, {\em
  An asymptotically compatible coupling formulation for nonlocal interface
  problems with jumps}, arXiv preprint arXiv:2203.07565,  (2022).

\bibitem{gorenflo2007continuous}
{\sc R.~Gorenflo, F.~Mainardi, and A.~Vivoli}, {\em Continuous-time random walk
  and parametric subordination in fractional diffusion}, Chaos Solitons
  Fractals, 34 (2007), pp.~87--103.

\bibitem{grubb2015fractional}
{\sc G.~Grubb}, {\em Fractional {L}aplacians on domains, a development of
  {H}\"{o}rmander's theory of {$\mu$}-transmission pseudodifferential
  operators}, Adv. Math., 268 (2015), pp.~478--528.

\bibitem{hackbusch2002h2}
{\sc W.~Hackbusch and S.~B\"{o}rm}, {\em {${\mathscr H}^{2}$}-matrix
  approximation of integral operators by interpolation}, Appl. Numer. Math., 43
  (2002), pp.~129--143.
\newblock 19th Dundee Biennial Conference on Numerical Analysis (2001).

\bibitem{hao2021fractional}
{\sc Z.~Hao, Z.~Zhang, and R.~Du}, {\em Fractional centered difference scheme
  for high-dimensional integral fractional {L}aplacian}, J. Comput. Phys., 424
  (2021), pp.~Paper No. 109851, 17.

\bibitem{huang2014numerical}
{\sc Y.~Huang and A.~Oberman}, {\em Numerical methods for the fractional
  {L}aplacian: a finite difference--quadrature approach}, SIAM J. Numer. Anal.,
  52 (2014), pp.~3056--3084.

\bibitem{jia2021fast}
{\sc J.~Jia, H.~Wang, and X.~Zheng}, {\em A fast collocation approximation to a
  two-sided variable-order space-fractional diffusion equation and its
  analysis}, J. Comput. Appl. Math., 388 (2021), pp.~Paper No. 113234, 14.

\bibitem{karkulik2019h}
{\sc M.~Karkulik and J.~M. Melenk}, {\em {$\mathcal{H}$}-matrix approximability
  of inverses of discretizations of the fractional {L}aplacian}, Adv. Comput.
  Math., 45 (2019), pp.~2893--2919.

\bibitem{levendorskii2004pricing}
{\sc S.~Z. Levendorski\u{\i}}, {\em Pricing of the {A}merican put under
  {L}\'{e}vy processes}, Int. J. Theor. Appl. Finance, 7 (2004), pp.~303--335.

\bibitem{Li2020}
{\sc X.~Li, Z.~Mao, N.~Wang, F.~Song, H.~Wang, and G.~E. Karniadakis}, {\em A
  fast solver for spectral elements applied to fractional differential
  equations using hierarchical matrix approximation}, Computer Methods in
  Applied Mechanics and Engineering, 366 (2020), p.~113053.

\bibitem{LIAN2016}
{\sc Y.~Lian, Y.~Ying, S.~Tang, S.~Lin, G.~J. Wagner, and W.~K. Liu}, {\em A
  petrov--galerkin finite element method for the fractional
  advection--diffusion equation}, Computer Methods in Applied Mechanics and
  Engineering, 309 (2016), pp.~388--410.

\bibitem{lindgren2022spde}
{\sc F.~Lindgren, D.~Bolin, and H.~v. Rue}, {\em The {SPDE} approach for
  {G}aussian and non-{G}aussian fields: 10 years and still running}, Spat.
  Stat., 50 (2022), pp.~Paper No. 100599, 29.

\bibitem{lindgren2011explicit}
{\sc F.~Lindgren, H.~v. Rue, and J.~Lindstr\"{o}m}, {\em An explicit link
  between {G}aussian fields and {G}aussian {M}arkov random fields: the
  stochastic partial differential equation approach}, J. R. Stat. Soc. Ser. B
  Stat. Methodol., 73 (2011), pp.~423--498.
\newblock With discussion and a reply by the authors.

\bibitem{luchko2016new}
{\sc Y.~Luchko}, {\em A new fractional calculus model for the two-dimensional
  anomalous diffusion and its analysis}, Math. Model. Nat. Phenom., 11 (2016),
  pp.~1--17.

\bibitem{malhotra2015pvfmm}
{\sc D.~Malhotra and G.~Biros}, {\em P{VFMM}: a parallel kernel independent
  {FMM} for particle and volume potentials}, Commun. Comput. Phys., 18 (2015),
  pp.~808--830.

\bibitem{minden2020simple}
{\sc V.~Minden and L.~Ying}, {\em A simple solver for the fractional
  {L}aplacian in multiple dimensions}, SIAM J. Sci. Comput., 42 (2020),
  pp.~A878--A900.

\bibitem{sauter2010boundary}
{\sc S.~A. Sauter and C.~Schwab}, {\em Boundary element methods}, vol.~39 of
  Springer Series in Computational Mathematics, Springer-Verlag, Berlin, 2011.
\newblock Translated and expanded from the 2004 German original.

\bibitem{silling2000reformulation}
{\sc S.~A. Silling}, {\em Reformulation of elasticity theory for
  discontinuities and long-range forces}, J. Mech. Phys. Solids, 48 (2000),
  pp.~175--209.

\bibitem{SONG2016}
{\sc F.~Song, C.~Xu, and G.~E. Karniadakis}, {\em A fractional phase-field
  model for two-phase flows with tunable sharpness: Algorithms and
  simulations}, Computer Methods in Applied Mechanics and Engineering, 305
  (2016), pp.~376--404.

\bibitem{suzuki2022fractional}
{\sc J.~L. Suzuki, M.~Gulian, M.~Zayernouri, and M.~D'Elia}, {\em Fractional
  modeling in action: a survey of nonlocal models for subsurface transport,
  turbulent flows, and anomalous materials}, Journal of Peridynamics and
  Nonlocal Modeling,  (2022).

\bibitem{WANG2015}
{\sc H.~Wang, D.~Yang, and S.~Zhu}, {\em A petrov--galerkin finite element
  method for variable-coefficient fractional diffusion equations}, Computer
  Methods in Applied Mechanics and Engineering, 290 (2015), pp.~45--56.

\bibitem{darve21}
{\sc R.~Wang, C.~Chen, J.~Lee, and E.~Darve}, {\em Pbbfmm3d: A parallel
  black-box algorithm for kernel matrix-vector multiplication}, Journal of
  Parallel and Distributed Computing, 154 (2021), pp.~64--73.

\bibitem{WANG2019}
{\sc T.~Wang, F.~Song, H.~Wang, and G.~E. Karniadakis}, {\em Fractional
  gray--scott model: Well-posedness, discretization, and simulations}, Computer
  Methods in Applied Mechanics and Engineering, 347 (2019), pp.~1030--1049.

\bibitem{wang2021exafmm}
{\sc T.~Wang, R.~Yokota, and L.~A. Barba}, {\em Exafmm: a high-performance fast
  multipole method library with c++ and python interfaces}, Journal of Open
  Source Software, 6 (2021), p.~3145.

\bibitem{XU2020}
{\sc K.~Xu and E.~Darve}, {\em Isogeometric collocation method for the
  fractional laplacian in the 2d bounded domain}, Computer Methods in Applied
  Mechanics and Engineering, 364 (2020), p.~112936.

\bibitem{ying2004kernel}
{\sc L.~Ying, G.~Biros, and D.~Zorin}, {\em A kernel-independent adaptive fast
  multipole algorithm in two and three dimensions}, J. Comput. Phys., 196
  (2004), pp.~591--626.

\bibitem{yokota14}
{\sc R.~Yokota, G.~Turkiyyah, and D.~Keyes}, {\em Communication complexity of
  the fast multipole method and its algebraic variants}, Supercomput. Front.
  Innov.: Int. J., 1 (2014), pp.~63--84.

\bibitem{H2Opus}
{\sc S.~Zampini, W.~Boukaram, G.~Turkiyyah, O.~Knio, and D.~Keyes}, {\em
  H2{O}pus: a distributed-memory multi-{GPU} software package for non-local
  operators}, Adv. Comput. Math., 48 (2022), pp.~Paper No. 31, 32.

\bibitem{zhang2007numerical}
{\sc H.~Zhang, F.~Liu, and V.~Anh}, {\em Numerical approximation of
  {L}\'{e}vy-{F}eller diffusion equation and its probability interpretation},
  J. Comput. Appl. Math., 206 (2007), pp.~1098--1115.

\bibitem{ZHAO2019}
{\sc T.~Zhao, Z.~Mao, and G.~E. Karniadakis}, {\em Multi-domain spectral
  collocation method for variable-order nonlinear fractional differential
  equations}, Computer Methods in Applied Mechanics and Engineering, 348
  (2019), pp.~377--395.

\bibitem{ZHAO2017}
{\sc X.~Zhao, X.~Hu, W.~Cai, and G.~E. Karniadakis}, {\em Adaptive finite
  element method for fractional differential equations using hierarchical
  matrices}, Computer Methods in Applied Mechanics and Engineering, 325 (2017),
  pp.~56--76.

\bibitem{zheng2020optimal}
{\sc X.~Zheng and H.~Wang}, {\em An optimal-order numerical approximation to
  variable-order space-fractional diffusion equations on uniform or graded
  meshes}, SIAM J. Numer. Anal., 58 (2020), pp.~330--352.

\end{thebibliography}

\end{document}